%% file: geometric_routing_arXiv.tex
\documentclass[a4paper,UKenglish,cleveref, autoref, thm-restate]{lipics-v2021}

\pdfoutput=1 
\hideLIPIcs  

\usepackage{cite}
\usepackage{amssymb}
\usepackage{mathtools}
\usepackage{algorithm}
\usepackage{algorithmicx}
\usepackage[noend]{algpseudocode}
\input{macrosetup.tex}

\newtheorem*{theorem*}{Theorem}

\title{Geometric Routing in Geometric Inhomogeneous Random Graphs} 

\author{Yu-Cheng Chiu}{Department of Mathematics, ETH Z\"urich, Switzerland}{ycchiu222@gmail.com}{https://orcid.org/0009-0002-9640-8221}{}
\author{Marc Kaufmann}{Department of Computer Science, ETH Z\"urich, Switzerland}{marc.kaufmann@inf.ethz.ch}{https://orcid.org/0000-0001-8489-2058}{Swiss National Science Foundation grant number 200021\_192079}
\author{Kostas Lakis}{Department of Computer Science, ETH Z\"urich, Switzerland}{klakis@ethz.ch}{https://orcid.org/0009-0004-5595-1839}{Swiss National Science Foundation grant number 10003390}
\author{Ulysse Schaller}{Department of Computer Science, ETH Z\"urich, Switzerland}{ulysse.schaller@inf.ethz.ch}{https://orcid.org/0009-0002-1416-1086}{Swiss National Science Foundation grant number 200021\_192079}

\authorrunning{Y.C. Chiu, M. Kaufmann, K. Lakis, U. Schaller} 

\Copyright{Yu-Cheng Chiu, Marc Kaufmann, Kostas Lakis, Ulysse Schaller} 


 
\acknowledgements{We would like to thank Johannes Lengler for suggesting this problem to us and for the insightful conversations about it.}

\nolinenumbers 

\begin{document}

\maketitle

\begin{abstract}
We present the first rigorous analysis of decentralized geometric routing in Geometric Inhomogeneous Random Graphs (GIRGs), a weight-agnostic variant of the greedy routing protocol. While greedy routing in GIRGs is known to explain the \emph{algorithmic small-world phenomenon} by finding ultra-short paths of length \(\Theta(\log \log n)\), it assumes additional knowledge of vertex weights beyond geometry, an assumption that is often restrictive or unavailable. We investigate whether the underlying geometry alone is sufficient for efficient navigation. We prove that for power-law weight exponent \(\tau \in (2,3)\) and geometric decay parameter \(\alpha > \tau-1\), geometric routing succeeds with constant probability and finds ultra-short paths of length \(\Theta(\log \log n)\), matching the optimal asymptotic guarantees for greedy routing. Our analysis further reveals that, upon success, both protocols follow a similar two-phase trajectory, consisting of a rapid ascent to the heavy vertices, followed by efficient navigation to the target. These results demonstrate that, in the appropriate regime, the network’s geometry alone implicitly guides the path to the target through its high-weight core.
\end{abstract}

\section{Introduction}
The phenomenon that social networks contain surprisingly short paths, and admit decentralized navigation, was popularized by  Milgram's ``small-world'' experiment \cite{M67,TM77}. This motivated the study of decentralized routing algorithms on random graph models of complex networks.
Kleinberg's seminal work~\cite{K00-1,K00} provided the first theoretical explanation of the algorithmic small-world phenomenon by showing that greedy routing succeeds in \(\Oh{\log^2 n}\) steps on augmented grids. Despite its influence, Kleinberg's model has several shortcomings, such as fragile parameter choices, perfect lattice structure, and homogeneous degrees. In particular, it yields a rather large stretch compared to the average distance. While a long line of works \cite{MN04,NM05,FM06,FGP04,FG10,FG14} has extended Kleinberg’s framework to more general settings, none has addressed all deficiencies simultaneously. See \cite{BKLMM22} for a comprehensive discussion.

To overcome these deficiencies, we focus on \emph{Geometric Inhomogeneous Random Graphs} (GIRGs)~\cite{BKL19}, which have emerged as a compelling framework for modeling scale-free networks. In this model, vertices are distributed independently and uniformly at random in a geometric space, each assigned a weight drawn from a power-law distribution. The connection probability combines two effects: it increases with vertex weights (favoring hubs) and decreases with geometric distance (favoring locality). A formal definition is provided in \Cref{subsec:GIRG_model}.

The study of \emph{greedy routing} in GIRGs \cite{BKLMM22} provides rigorous theoretical explanations for the algorithmic small-world phenomenon. Lacking global knowledge of the network topology, the strategy is to maximize, at each step, the connection probability to the target using only local neighbor information. An intrinsic challenge in inhomogeneous geometric graphs is that an ideal neighbor must balance proximity to the target \(d(v,t)\) with vertex weight \(w_v\). In \(d\)-dimensional GIRGs this trade-off is captured by the objective function \(\phi(v) \coloneqq \frac{w_v}{d(v,t)^d}\), naturally derived from its edge probability. Bringmann et al.~\cite{BKLMM22} obtained:
\begin{itemize}
\item Greedy routing succeeds with probability \(\Om{1}\).
\item Conditioned on success, the number of steps is at most \( \frac{2 + \oh{1}}{\abs{\log(\tau-2)}} \log\log n \) a.a.s.
\end{itemize}
Both results are asymptotically optimal. Specifically, the number of steps matches the average distance in the giant component, yielding a stretch of \(1+\oh{1}\).

Towards practical realizability, they further showed that greedy routing is robust under relaxations, an approximation of the objective function \(\phi(v)\) suffices for success with optimal stretch. However, in realistic settings like social networks or Internet routing, even approximate weight information might be inaccessible due to privacy concerns or limited storage.
We therefore address these limitations by studying routing processes purely based on geometry.

Motivated by this, we study \emph{geometric routing} in GIRGs. While greedy routing prioritizes connection probability, geometric routing acts greedily on geometry alone, simply forwarding to the neighbor closest to the target in the geometric space. This approach is inherently robust in real-world scenarios, circumventing issues of inaccurate popularity estimates (as in Milgram’s experiment) and adapting naturally to dynamic Internet-like graphs, where weights fluctuate over time and new nodes join, while the underlying geometry remains fixed.

Despite these advantages, geometric routing faces fundamental challenges. Without explicit weight information, the protocol must rely solely on the implicit geometry-weight correlation to find short paths via the heavy core. Consequently, its success guarantees and efficiency remain unclear.
A central goal of this paper is to analyze whether this weight-agnostic approach can match the performance of greedy routing.

\paragraph*{Our contributions}

We give the first rigorous analysis of geometric routing in GIRGs. We focus on the regime with power-law exponent \(2<\tau<3\) and decay parameter \(\alpha>\tau-1\), where larger $\alpha$ corresponds to a sharper geometric
decay in the GIRG edge probabilities, we prove:
\begin{itemize}
\item Geometric routing succeeds with probability \(\Om{1}\). 
\item Conditioned on success, the main part takes at most \( \frac{2 + \oh{1}}{\abs{\log(\tau-2)}} \log\log n \) steps a.a.s.
\end{itemize}

Remarkably, in this regime, geometric routing matches the asymptotically optimal performance of greedy routing, yielding a stretch of \(1+o(1)\) a.a.s., conditioned on success. We further reveal that both protocols exhibit a similar two-phase structure:
\begin{enumerate}
\item{\textbf{First phase: ascent to the heavy core.}} 

For both protocols, vertex weights grow doubly exponentially with rate \(1/(\tau-2)\).

\item{\textbf{Second phase: navigation from the core to the target.}}
\begin{itemize}
\item \emph{Geometric routing}: Requires a specific \emph{preparation step} (combining weight growth and distance reduction) at the phase transition, before the geometric distance to the target decays doubly exponentially with rate \(\tau-2\).
\item \emph{Greedy routing}: Proceeds seamlessly without an intermediate step. The objective function \(\phi(v)\) simply continues to grow doubly exponentially with rate \(\tau-2\).
\end{itemize}
\end{enumerate}

\subsection{Related Works}
\paragraph*{Mean field analysis and simulations of geometric routing}
The idea of weight-agnostic geometric routing\footnote{In \cite{BK09}, this is termed ``greedy routing''. We use geometric routing to avoid confusion.} was previously explored by Bogun{\'a} and Krioukov~\cite{BK09} in scale-free networks embedded in a homogeneous isotropic \(D\)-dimensional metric space.
Using mean-field analysis for the first phase of the process, they concluded that, in our terminology, the navigable regime corresponds to \(\alpha> 1+1/d\).
Bogun{\'a} et al. further performed numerical simulations and comparisons with real-world data \cite{BK09,BKC09}, observing that for a fixed clustering strength (governed by \(\alpha\)), there exists a critical power-law exponent \(\tau\) that determines navigability as network size grows.

This observation and other empirical findings of Boguná et al.~\cite{BKC09} closely aligned with our theoretical results for GIRGs. In particular, their simulations demonstrated that geometric routing succeeds with a positive fraction of pairs in the strong-clustering (large \(\alpha\)) and heavy-tailed (small \(\tau\)) regime. This matches our result that geometric routing in GIRGs succeeds with constant probability when \(\alpha>\tau-1\). Moreover, they observed that within the navigable regime, the probability of hopping to a higher-degree neighbor remains high until reaching a distance-dependent critical degree, after which it decays rapidly. This observation corroborates our proposed two-phase trajectory via the heavy core. Finally, they reported that the average path length is polylogarithmic in the network size. Their simulations appear consistent with our finding that successful geometric routing in GIRGs requires \(\Th{\log \log n}\) steps, noting that the distinction between \(\log\log n\) and more general polylogarithmic bounds may not be noticeable for the system sizes considered in their experiments.

\paragraph*{Hyperbolic random graphs (HRGs) and hyperbolic embeddings}
A prominent special case of GIRGs, \emph{hyperbolic random graphs} (HRGs), have been extensively validated in experiments as a realistic model of real-world networks \cite{ST08,BKC09,BPK10,KPKVB10,PPK14,GBAS16,GASB19,ZSZZ11,LGZZAKW16}.
HRGs originated in Robert Kleinberg's\footnote{Not to be confused with Jon Kleinberg’s small-world model \cite{K00}.} work
on routing in ad hoc wireless networks \cite{K07}. Hyperbolic geometry turned out to work well for scale-free networks \cite{F19}. Subsequent research on HRGs includes theoretical foundation \cite{KPKVB10}, experimental validations of greedy processes \cite{KPBV09,PKBV10,CC09}, and empirical evidence for navigability in real-world networks \cite{BKC09,BPK10}.

Bringmann et al.~\cite{BKL19,BKLMM22} formally established that HRGs are essentially a special case of GIRGs, and the ``greedy geometric processes'' studied in HRGs \cite{BKC09,BPK10,KPBV09,PKBV10,CC09} are essentially greedy routing in 1-dimensional GIRGs.
In HRGs, nodes are typically embedded in a 2-dimensional hyperbolic disk\footnote{For recent progress on higher-dimensional hyperbolic embeddings, see Jankowski et al.~\cite{JABS23}.} without assigned weight. The radial and angular coordinates represent popularity and similarity, which translate naturally to weight and 1-dimensional distance in GIRGs. 
GIRGs extend this to arbitrary dimension and avoid the complexity of hyperbolic trigonometry. HRGs are known to be sparse and scale-free \cite{GPP12}, have constant clustering coefficient \cite{CF16,GPP12}, polylogarithmic diameter \cite{FK18,KM14}, and a giant component \cite{BFM15,FM18}. Bringmann et al.~\cite{BKL17,BKL19,BKL25} proved that GIRGs inherit all aforementioned properties.

Driven by the model's realism, recent studies have focused on developing hyperbolic embeddings with specific routing-oriented properties \cite{BFKL18,GASB19,BFKK20,JABS23, JLZFS22,LLRM24,YJJWH22,SP23,ASB25}. Garc{\'\i}a-P{\'e}rez et al.~\cite{GASB19} introduced \emph{Mercator}, a fast, accurate, and robust algorithm for embedding real-world networks into the hyperbolic plane. Bl{\"a}sius et al.~\cite{BFKK20} proposed embeddings that guarantee the success of greedy routing while maintaining low stretch. More recently, Jankowski et al.~\cite{JABS23} studied \(D\)-Mercator, extending Mercator to higher-dimensional hyperbolic space, and evaluated whether higher-dimensional embeddings are more powerful using several criteria, including the success rate of geometric routing.

\paragraph*{Routing in generalized GIRGs:}
Sevilla and Fern\'andez Anta~\cite{SF20} studied routing protocols in GIRGs and the generalized \((\kappa,\pi)\)-KG model, which interpolates between Kleinberg's model and scale-free networks.
They analyzed a modified greedy objective that introduces a ``critical ball'' of radius \(r\) around the target, within which routing purely minimizes geometric distance. Their simulations show that increasing \(r\) improves success probability at the cost of slightly longer paths.

\section{Preliminaries}
\subsection{Geometric Inhomogeneous Random Graphs (GIRGs)}\label{subsec:GIRG_model}
A \emph{Geometric Inhomogeneous Random Graph (GIRG)} is a graph \(G = (V,E)\) defined over a geometric space \(\mathcal{X}\) as follows:

\emph{Geometric space:}
We consider a \(d\)-dimensional torus \(\mathcal{X}\) of volume \(n\), which is essentially a \(d\)-dimensional cube \([0, n^{\frac{1}{d}}]^d\) with torus topology. We consider the \(\infty\)-norm on \(\mathcal{X}\), i.e.\ \(\forall x,y\in\mathcal{X},\,\,\norm{x-y}_{\infty} \coloneqq \max_{i\in[d]} \left\{\min \left(\abs{x_i - y_i},n^{\frac{1}{d}}-\abs{x_i-y_i} \right)\right\}\). The choice of the \(\infty\)-norm is a matter of technical convenience, as the resulting model would remain the same if any other norm were used, as any constant factor deviation is absorbed by the asymptotic notation used in the edge probability definitions below. For the remainder of this paper, we will omit the subscript \(\infty\) from the norm and simply write \(\norm{x - y}\).

\emph{Vertices:}
The set of vertices \(V\) is given by a Poisson point process on \(\mathcal{X}\) with intensity \(1\). In particular, the expected number of vertices is \(n\), and the positions of all vertices are chosen uniformly at random in \(\mathcal{X}\). For each measurable \(\mathcal{A} \subseteq \mathcal{X}\), let \(V(\mathcal{A})\) denote the set of vertices in \(\mathcal{A}\), then \(\abs{V(\mathcal{A})}\) is Poisson distributed with mean \(\textrm{Vol}(\mathcal{A})\). Moreover, if \(A,B \subseteq \mathcal{X}\) are disjoint measurable sets, then \(\abs{V(\mathcal{A})}\) and \(\abs{V(\mathcal{B})}\) are independent Poisson variables.

\emph{Weights:} 
Let \(2 < \tau < 3\) be a fixed constant. Each vertex \(v\in V\) independently receives a weight \(w_v\) drawn from a probability distribution with density
\[
f(w) \coloneqq
\begin{cases}
\Th{w^{-\tau}} & \text{if } w \geq 1,\\
0 & \text{if } w < 1.
\end{cases}
\]
That is, a power-law distribution with exponent \(\tau\) and minimal weight\footnote{More generally, \(w_{\min}\) can be a tunable scale parameter. We fix \(w_{\min} = 1\) as is standard, and restoring it to any fixed constant does not introduce significant difficulties to our analysis.} \(1\).

\emph{Edges:} Given the vertex set \(V\), positions \(\{x_v\}_{v\in V}\), and weights \(\{w_v\}_{v\in V}\), for some constant \(\alpha>1\), we connect each pair of distinct vertices \(u,v\) independently with probability
\begin{equation}\tag{EP1}\label{eq:EP1}
p_{uv} \coloneqq
\Th*{
\min\left\{ \Big( \frac{w_u w_v}{\norm{x_u - x_v}^d}\Big)^\alpha, 1
\right\}
}.
\end{equation}
The parameter \(\alpha\) governs the sharpness of the decay below the natural weight-distance threshold, as determined by the ratio \(w_u w_v/\norm{x_u-x_v}^d\).
In \eqref{eq:EP1}, there exists a constant \(c_1 >0\) such that \(p_{uv} = \Theta(1)\) whenever \(u\) and \(v\) are sufficiently close, i.e.\ when \(\norm{x_u - x_v}^d \leq c_1 w_u w_v\). In the \emph{threshold case} \(\alpha = \infty\), the connection probability drops to zero as soon as the vertices are too far apart. Instead of \eqref{eq:EP1}, we require constants \(0 < c_1 \leq c_2\) such that
\begin{equation}\tag{EP2}\label{eq:EP2}
p_{uv}=
\begin{cases}
\Th*{1} & \text{if } \norm{x_u - x_v}^d \leq c_1 w_u w_v,\\
0 & \text{if } \norm{x_u - x_v}^d \geq c_2 w_u w_v.
\end{cases}
\end{equation}
Note that there can be an interval \(\norm{x_u - x_v} \in \left[(c_1 w_u w_v)^{1/d}, (c_2 w_u w_v)^{1/d}\right]\) in which the behavior of \(p_{uv}\) is arbitrary.

In summary, the free parameters of the model are the expected number of vertices \(n\), the power-law parameter \(\tau\), the dimension \(d \in \mathbb{N}\), the decay parameter \(\alpha>1\), and the specific probability function \(p_{uv}\) that satisfies \eqref{eq:EP1} or \eqref{eq:EP2}. In all subsequent theorems and proofs, we assume that all parameters except \(n\) are fixed constants. Consequently, all asymptotic notation is with respect to \(n \to \infty\).

\subsection{Properties of GIRGs}
We next state some basic properties of GIRGs, see \cite{BKL19,BKL25} for proofs.

\begin{lemma}\label{lem:edge_marginal}
Let \(u,v\) be two vertices of a GIRG. Then
\[
\mathop{\mathbb{P}}\limits_{x_u,x_v}\!\big[ \{u,v\}\in E \,\big|\, w_u,w_v \big] = \Th*{ \min\left\{ \frac{w_u w_v}{n},\, 1 \right\} }.
\]
\end{lemma}
Fixing one of \(x_u\) or \(x_v\) (but not both) does not change the marginal probability by more than a constant factor. Since the vertices form a Poisson point process, this immediately implies:

\begin{lemma}\label{lem:deg_poisson}
Let \(v \in V\) be a fixed vertex with weight \(w_v\). Then the random variable \(\deg(v)\) follows a Poisson distribution with expectation \(\E{\deg(v)} = \Th{w_v}\).
\end{lemma}

The next lemma from \cite{BKL25} shows that GIRGs contain a unique giant component of linear size, and that the average distance inside it is \(\Th*{\log \log n}\).
This will later be seen to match the number of steps taken by greedy and geometric routing up to a \(1+\oh{1}\) stretch.
\begin{lemma}\label{lem:average_dist}
A.a.s. the largest component has linear size, while all other components have size \(\log^{\Oh{1}} n\).
And a.a.s. the average distance in the largest component is 
\(\frac{2\pm \oh{1}}{\abs{\log(\tau-2)}}\,\log\log n \).
\end{lemma}

\subsection{Greedy Routing and Geometric Routing}\label{subsec:routing_protocols}
Both geometric and greedy routing protocols in GIRGs follow a local greedy strategy: the current vertex \(v\) forwards the packet to a neighbor \(u \in N(v)\) that optimizes a specific criterion, stopping if no neighbor improves the objective (failure) or if target \(t\) is reached (success).

\emph{Greedy routing protocol}: Every vertex \(v\) knows its own address \((x_v,w_v)\), those of its neighbors, and \(t\)'s address is enclosed with the packet. In each hop, the protocol forwards the packet to a neighbor maximizing the connection probability to \(t\), which is equivalent to maximizing the objective function \(\phi(u) \coloneqq \frac{w_u}{\norm{x_u - x_t}^d}\) for \(u \in N(v)\).

In their paper, Bringmann et al.~\cite{BKLMM22} established the following theorems:
\begin{theorem*}[Success Probability]
Greedy routing succeeds with probability \(\Om{1}\).
\end{theorem*}

\begin{theorem*}[Greedy Path Length]\label{thm:greedy_length}
Conditioned on success, greedy routing finds a path of length at most \(\frac{2+\oh{1}}{\abs{\log(\tau-2)}} \log \log n\) asymptotically almost surely.
\end{theorem*}
Comparing the length of the greedy path and the average distance in \Cref{lem:average_dist}, we get

\begin{theorem*}[Stretch]
A.a.s. the stretch of a successful greedy path is \(1+\oh{1}\).
\end{theorem*}
A success probability of \(\Th{1}\) and stretch of \(1+\oh{1}\) are asymptotically best possible. In GIRGs, success is inherently limited by the model's probabilistic nature: even for vertices at infinitesimal distance, the connection probability remains \(\Th{1}\). Thus, any local routing may fail at the final hop if the current vertex is not adjacent to \(t\), despite being arbitrarily close.

The greedy routing protocol depends on both the vertex weight \(w_u\) and geometric distance \(\norm{x_u - x_t}\), which poses practical restrictions. In this paper, we study the \emph{geometric routing protocol} in GIRGs, which ignores vertex weights and relies solely on geometric positions.

\emph{Geometric routing protocol:}  Every vertex \(v\) knows its own position \(x_v\), its neighbors' positions, and \(x_t\) is enclosed with the packet. In each hop, the protocol forwards the packet to a neighbor \(u\) closest to \(x_t\).


We note that the target vertex \(t\) globally maximizes \(\phi\) and minimizes the distance to \(x_t\), providing a natural terminus for both greedy and geometric routing. The sequence of vertices obtained by applying geometric (resp.\ greedy) routing from source \(s\) to target \(t\) is called the \emph{geometric path} (resp.\ \emph{greedy path}).

\section{Results}
\subsection{Geometric routing in GIRGs}
Geometric routing operates on a simpler heuristic than greedy routing, yet its performance in GIRGs requires further analysis. In GIRGs, ultra-short paths of order \(\Th{\log\log n}\) typically arise from traversing a "heavy core" of high-weight vertices that act as network hubs~\cite{BKL25}. Without knowledge of weights, it is a priori unclear whether the protocol can effectively navigate through these hubs, a process essential for efficient routing. We show that the implicit correlation between geometry and connectivity is sufficient for finding these shortcuts. Under suitable parameters, geometric routing achieves the same performance as greedy routing, succeeding with constant probability and maintaining \(\Theta(\log\log n)\) path lengths.

\subsubsection*{The regime \(\alpha>\tau-1\):}
The success and efficiency of geometric routing are parameter-dependent, determined by the interplay between the network's geometry and its scale-free structure. Performance is governed by two key factors: (i)~the power-law exponent \(\tau\) for the weight distribution; and (ii)~the parameter \(\alpha\), which
controls how sharply connection probabilities decay geometrically.

Intuitively, \(\tau\) dictates the prevalence of extremely heavy vertices: smaller \(\tau\) corresponds to a heavier tail, while larger \(\tau\) yields a more homogeneous weight distribution. The parameter \(\alpha\) controls the influence of geometry on edge formation. In particular, it determines the fraction of ``weak ties''~\cite{G73} (long-range edges between low-weight nodes). Larger \(\alpha\) makes weak ties rare (the threshold case \(\alpha = \infty\) forbids them completely), resulting in highly clustered networks where long-range neighbors are predominantly heavy. Conversely, smaller  \(\alpha\) weakens geometric constraints, resulting in a more random network structure. 
 
For geometric routing to be efficient, neighbors selected for their proximity to \(t\) should be sufficiently heavy (i.e., satisfy a weight lower bound). This property is captured by the condition \(\alpha > \tau-1\). If \(\alpha < \tau - 1\), the most distant neighbors may be weak ties of low weight with non-negligible probability, which impedes progress toward the heavy core and causes routing to stall in light vertices. Henceforth, we assume \( \alpha>\tau-1\) unless specified otherwise.

\subsubsection*{Main results:}
Let \(G\) be a GIRG with power-law weight exponent \(2 < \tau < 3\) and decay parameter \(\alpha > \tau - 1\). Typically, source \(s\) and target \(t\) are sampled uniformly at random, but all results also hold if we condition on specific values for \(x_s , x_t , w_s , w_t\). Thus, we assume \(s\) and \(t\) are fixed throughout the analysis, while the remaining vertices and all edges are drawn randomly.
Our main result shows that geometric routing from \(s\) to \(t\) in \(G\) succeeds with constant probability achieves \(\Theta(\log \log n)\) path lengths. The main proof ideas are given in \Cref{subsec:proof_sketch}; the full
formal proofs are deferred to the full version.

\begin{restatable}{theorem}{MainTheorem}\label{thm:main}
Geometric routing succeeds with probability \(\Om{1}\).
\end{restatable}

As with greedy routing, \(\Th{1}\) success reflects the model's inherent connectivity, being optimal in the general setting. Conditioned on success, geometric routing is efficient in traversing the main part of the routing process, as the following theorem shows.

\begin{theorem}\label{thm:length}
Let \(w_0, D_0\) be sufficiently large constants. The geometric routing protocol finds a path from a vertex of weight at least \(w_0\) to a vertex within distance \(D_0\) of \(t\), which constitutes the main part of the routing, whose length is upper-bounded by
\begin{equation}\label{eq:length}
\frac{1 + \oh{1}}{\abs{\log(\tau - 2)}} \left( \log \log_{w_0} (D_{s}^d) + \log \log_{D_0} (D_{s}) \right) + \Oh{f_0(n)},
\end{equation}
where \(D_s = \norm{x_s -x_t}\) and \(\om{1} \leq f_0(n) \leq o(\log\log n)\). As \( D_s \leq n^{1/d}\), the bound evaluates to \[\frac{2+\oh{1}}{\abs{\log(\tau-2)}}\log \log n.\]
\end{theorem}
For randomly \(s\) and \(t\) (typically having constant weights and are geometrically far apart), both greedy and geometric routing achieve path length \(\frac{2+\oh{1}}{\abs{\log(\tau-2)}}\log \log n\) for the main part of the routing.
This demonstrates that navigation based solely on geometry is near-optimal for \(\alpha>\tau-1\), matching the efficiency of  greedy routing.

\subsection{Typical Trajectory}
We outline the geometric routing process and the intuition behind our analysis, highlighting key differences from greedy routing. 
Formal statements are provided in \Cref{subsec:proof_sketch}, along with the main proof ideas. Complete proofs are given in the appendices.

\subsubsection{Typical Trajectory of a Geometric Path}\label{subsec:typical_traj_geometric}
Under \(\alpha>\tau-1\), geometric proximity implicitly governs weight dynamics, guiding the path through a predictable ascent to the core and subsequent descent to \(t\). The expected trajectory of the routing process is such that, after a few steps from the start, the current weight becomes a sufficiently large constant \(w_0\). We also expect that, near the end of the routing, when the distance to the target becomes some constant, only a few additional hops are needed to reach the target.
We refer to the phase during which the routing process first reaches a vertex of sufficiently large constant weight \(w_0\) as the \emph{start phase}, and the phase during which the distance to the target are below a sufficiently large constant \(D_0\) as the \emph{end phase}. Our primary analysis focuses on the \emph{main part} of the routing, which occurs in between.
We will argue that, with high probability, the routing process does not fail and progresses efficiently toward the target during this main part. To this end, we leverage \cref{lem:deg_poisson}, which established \(\E{\deg(v)} = \Th{w_v}\) for any vertex \(v\), implying that the weights of its neighbors follows a power-law with exponent \(\tau-1\).

\paragraph*{First phase: implicit weight growth.}
For a low-weight vertex \(v\), a.a.s. all its neighbors are spatially local and light. Since the current vertex \(v\) is still far from \(t\), its neighbors are roughly equidistant to \(t\), up to factors \((1 + \oh{1})\). The protocol thus simply selects a neighbor \(u\) that: (i) lies in the direction of \(t\) from \(v\); and (ii) be away from \(v\) as far as possible.

For condition (i), any neighbor of \(v\) has a constant probability to lie in the ``same direction'' as \(t\) from \(v\), regardless of their weight. Condition (ii) encourages choosing a neighbor \(u\) with larger weight. Since \(\alpha> \tau-1\), weak ties are rare, most distant neighbors are ``strong neighbors'', having large enough weights to connect to \(v\) with probability \(\Om{1}\). We can estimate the expected number of strong neighbors at distance \(\approx r\), with weight \(\Omega(r^d/w_v)\), by
\[
\Th*{r^d \cdot \int_{r^d/w_v}^{\infty} w^{-\tau} dw} = \Th{r^{d(2-\tau)} w_v^{\tau-1}}.
\]
Let \(r_{\max}(v) \coloneqq w_v^{(\tau-1)/(d(\tau-2))}\) be the maximum \(r\) where \(\Om{1}\) strong neighbors are expected. Since a.a.s.\ there are no weak neighbors in this distance class of \(r_{\max}(v)\), the protocol typically selects a neighbor at distance \(\approx r_{\max}(v)\) from \(v\) with weight of order \(\Om{r_{\max}^d / w_v} = \Om{w_v^{1/(\tau-2)}}\). In other words, besides maximizing geometric progress toward \(t\), geometric routing exponentiates the current weight by  \(1/(\tau-2)\) at each step of the first phase. This phase ends when the process hits a vertex \(v\) heavy enough to satisfy \(r_{\max}(v) \approx D_v \coloneqq \norm{x_v - x_t}\). As a result, geometric routing reaches the heavy core within \(\frac{1+ \oh{1}}{\abs{\log (\tau-2)}} \log\log n\) steps.

\paragraph*{Second phase: rapid distance reduction.} 
Once the routing reaches a node \(v\) such that \(r_{\max}(v) \gtrsim D_v\), the relation between the current weight and the distance to \(t\) changes, since the length from \(v\) to its furthest neighbors can now exceed \(D_v\). The probability that \(v\) is directly adjacent to \(t\) remains low, but \(v\) has neighbors that are significantly closer to \(t\). Therefore, the routing protocol no longer hops simply to the furthest neighbors in the good direction. Instead, it seeks the neighbors of \(v\) in a small region around \(t\).

To capture the neighbors at distance roughly \(r\) from \(t\), consider an annulus \(\mathcal{A}(r)\) centered at \(t\) with outer radius \(r\) and inner radius \(r^{1-\Oh{\eps}}\). By the structure of GIRGs, both the distribution of the maximum-weight vertex in \(\mathcal{A}(r)\) and the probability that \(\mathcal{A}(r)\) contains a neighbor of \(v\) depend on the volume \(\Th{r^d}\) of \(\mathcal{A}(r)\). Let \(r_v\) be the minimal radius such that \(\mathcal{A}_v \coloneqq \mathcal{A}(r_v)\) contains a neighbor of \(v\). Suppose \(u\) is a neighbor of \(v\) closest to \(t\), then \(u \in \mathcal{A}_v\) and \(D_u =\Th{r_v}\) by construction. Moreover, \(w_u\) is, up to a constant factor, the maximum weight among vertices in \(\mathcal{A}_v\). Indeed, if \(w_{u}\) were not asymptotically maximal, vertices of such weight would be abundant enough to occur in a smaller annulus of radii \(r' < r_v\). Since the distance from \(v\) to these vertices in \(\mathcal{A}(r')\) is of the same order as the distance to \(u\), with high probability, \(v\) would be adjacent to at least one such vertex, contradicting the minimality of \(r_v\). Under the power-law distribution of weights, we thus obtain 
\[w_u \approx \Th*{r_v ^{d/(\tau-1)}} \approx \Th*{D_u^{d/(\tau-1)}}.\]
It is worth noting that this weight-distance relation need not hold for \(v\) itself, but is guaranteed for all subsequent nodes in the second phase. Specifically, for any node \(u\) in this phase (after the first hop), the routing process finds a neighbor in the annulus \(\mathcal{A}_u\) with radius \(r_u \approx D_u^{\tau-2}\).
In other words, throughout the second phase, except for the very first hop, the distance to \(t\) decreases by an exponent of \(\tau-2\) at each step. Consequently, in the second phase, geometric routing reaches a vertex at constant distance to \(t\) within \(\frac{1+ \oh{1}}{\abs{\log (\tau-2)}} \log\log n\) steps.

\subsubsection{Comparison with Trajectory of a Greedy Path}\label{subsec:typical_greedy}	
The typical trajectories of geometric and greedy routing both exhibit a two-phase structure and achieve the same asymptotic progress, yet the underlying mechanisms are different. In particular, the two phases of geometric routing are determined by the relation between
\[
r_{\max}(v) \coloneqq w_v^{(\tau-1)/(d(\tau-2))} \quad \text{and} \quad D_v \coloneqq \norm{x_v -x_t},
\]
while in greedy routing the two phases are characterized by the relation between (see also~\cite{BKLMM22})
\[
\phi(v) \coloneqq w_v/D_v^d \quad \text{and} \quad w_v^{-1/(\tau-2)}.
\]
Both descriptions encode the same underlying threshold, phrased in different parameters.

In the first phase, both protocols exhibit doubly exponential weight growth at rate \(1/(\tau-2)\). In the second phase, geometric routing exhibits doubly exponential decay of the distance \(D_v\), whereas greedy routing only ensures doubly exponential growth of its objective \(\phi(v)\). Notably, geometric routing requires a single preparatory step before the desired doubly exponential decay begins, a feature absent in greedy routing. Below, we briefly outline the two-phase trajectory of greedy routing.

\paragraph*{First phase of greedy routing:}
When \(v\) has small weight \(w_v\) compared to \(D_v\), neighbors are a.a.s.\ equidistant to \(t\) within factors of \((1 + \oh{1})\). Conversely, neighbor weights fluctuate by non-negligible constant factors (so do their distances from \(v\), but this does not affect the objective values). Consequently, weight fluctuations dominate the selection process, favoring hops to heavier neighbors.

By the \(\tau-1\) power-law tail of neighbors, the expected number of neighbors with weight at least \(w\) is \(\Th*{w_v w^{2-\tau}}\). Setting this to \(\Th{1}\) implies \(v\) has a neighbor of weight \(w \approx w_v^{1/(\tau-2)}\), yielding a doubly exponential weight growth at rate \(1/(\tau-2)\). This stops when a node \(v\) with \(\phi(v) \gtrsim w_v^{-1/(\tau-2)}\) is reached, as the relation of the current weight and distance to \(t\) changed. Hence, this first phase needs only \(\approx \log \log n / \abs{\log (\tau-2)}\) steps.

\paragraph*{Second phase of greedy routing:}
Beyond the threshold \(\phi(v) \gtrsim w_v^{-1/(\tau-2)}\), there exist neighbors of \(v\) that are much closer to \(t\).
Despite a typically lower weight relative to \(v\), their proximity compensates enough to increase the objective\footnote{Note that we can assume \(\phi(v)<1\), otherwise \(v\) is incident to \(t\) with probability \(\Om{1}\).} to \(\approx \phi(v)^{\tau-2}\). One can further show that the best neighbor keeps the routing in the second phase, so we can repeat the argument. Therefore, in the second phase, the objective increases by an exponent \(\tau-2\) at each step. After \(\approx \log \log n / \abs{\log (\tau-2)}\) rounds, the routing reaches a node \(v\) with objective \(\phi(v) = \Omega(1)\).

\subsection{Proof Overview}\label{subsec:proof_sketch}
In the formal proofs, we show that w.h.p. geometric routing adheres to the typical trajectory described above.
Although the best neighbor \(u\) of a single vertex \(v\) is predictable (see \Cref{subsubsec:where_neighbor} for the precise statements), a naive step-by-step analysis suffers from intractable dependencies: selecting \(u\) as the best neighbor reveals that no superior candidates exist in \(v\)'s neighborhood, which biases the environment for all subsequent steps.

To bypass these dependencies, we adapt the \emph{layer technique} of~\cite{BKLMM22}: we partition the vertex space into carefully crafted layers corresponding to the two phases of geometric routing. Uncovering these layers sequentially, rather than  tracking individual nodes on the path, allows us to rigorously track the trajectory's progress.

\subsubsection{Formalizing the Trajectory: The Layer Technique}\label{subsubsec:layers}
We now describe the layer argument used to formalize the intuitive behavior of geometric routing. The argument tracks the progress of the routing process through layers, while the local estimates in \Cref{subsubsec:where_neighbor} provide the probabilistic guarantees for moving from one layer to the next. 

The proof is organized around a trajectory lemma, which combines the layer construction with the local progress estimates on where to expect suitable neighbors. We state the lemma in a compressed form that captures the consequences needed for our main results, while the full layer-exposure argument is given in \Cref{appendix:geo_traj_thm}. The lemma shows that, once the routing process enters the main part, the protocol forces the path to leave the current layer and enter a subsequent one with high probability, except for the transition layer, which may be visited twice. Consequently, the path does not get stuck, visits each layer only a controlled number of times, and reaches the constant-distance region around the target within the desired number of steps. This yields \Cref{thm:length} and is the key ingredient for \Cref{thm:main}.

\paragraph*{Layers for geometric routing}
To formalize our discussion, we first introduce some technical notations. Throughout, we assume \(\alpha> \tau-1\) and \(2<\tau<3\).
For any \(\eps > 0\), set \(\gamma(\eps) \coloneqq \frac{1-\eps}{\tau-2}\).  In our proofs, \(\eps\) would be sufficiently small such that \(\gamma(\eps) > 1\).
Let \(N(v)\) denote the set of neighbors of \(v\), and \(D_v \coloneqq \norm{x_v - x_t}\). We also define
\(R_v(\eps) \coloneqq ( w_v ^{1+\gamma(\eps)} )^{1/d} \approx w_v^{(\tau-1)/d(\tau-2)}\) for sufficiently small \(\eps\), which is the expected distance from \(v\) to its furthest neighbor.

Let \(\eps_1 = \eps_1(\alpha,\tau)\) be a fixed constant which will be chosen sufficiently small in the proofs, in particular, \(\gamma(\eps_1)>1\). We categorize each vertex \(v\) based on whether its distance to the target is greater or less than the expected distance from \(v\) to its furthest neighbor, and define:
\begin{align*}
V_1 \coloneqq \SetB[\Big]{ v\in V}{D_v \geq R_v(\eps_1)}\quad \text{ and} \quad
V_2 &\coloneqq \SetB[\Big]{ v\in V}{D_v \leq R_v(\eps_1)}.
\end{align*}
The vertices in \(V_1\) and \(V_2\) behave differently under the geometric routing protocol, where vertices in \(V_1\) (resp. \(V_2\)) correspond to the first phase (resp. the second phase) of the routing.

Recall from discussion in \cref{subsec:typical_traj_geometric}, it typically takes only few steps for the source \(s\) to reach a vertex with constant weight \(w_0\), and similarly, for a vertex located within a ball of radius \(D_0\) around \(t\) to reach the target \(t\). These are referred to as the start and end phases of the routing, respectively. To this end, for any pair \((w,D)\) of weight and distance, we define
\[
\overline{V}(w, D) \coloneqq \SetB{v \in V_1}{\text{(1) : } w_v \geq w; \text{(2) : } D_v \geq D} \cup \SetB{ v \in V_2 }{D_v \geq D }.
\]
The main part of the routing process takes place within the set \(\overline{V}(w, D)\) for some suitable constant values of \(w,D\). In particular, we choose \(w_0\) and \(D_0\) sufficiently large so that the lemmata in \Cref{subsubsec:where_neighbor} can be applied to \(\overline{V}(w_0, D_0)\). Moreover, let \( V_{\leq D_0} \coloneqq \SetB{v\in V}{D_v \leq D_0}\) characterizes the end phase that will be treated separately. Our primary goal is to show that the routing survives this main part \(\overline{V}(w_0, D_0)\) with sufficiently high probability, and eventually reaches \(V_{\leq D_0}\), which is a small ball around target \(t\).

We analyze the structure of the geometric path in the main part by partitioning the set \(\overline{V}(w_0, D_0)\) into layers. We will define two classes of layers which correspond to the two phases: the layers \(A_{1,j}\) divide the area \(V_1 \cap \overline{V}(w_0, D_0)\) and are defined via weights, whereas the layers \(A_{2,j}\) divide the area \(V_2 \cap \overline{V}(w_0, D_0)\) and are defined via distances. The idea is to show that, with sufficiently high probability, there exists a node \(u_\ell\) on the geometric path \(P\) which has neighbors in \(V_{\leq D_0}\), such that until reaching \(u_\ell\), \(P\) visits every layer at most once, except possibly the special layer \(A_{2,\star}\), which may be visited twice. The special layer \(A_{2,\star}\) is the transition layer where the geometric path first enters the second phase.

Roughly speaking, we construct a sequence \(y_0=w_0<y_1<\cdots\) of weight thresholds and a sequence \(z_0=D_0<z_1<\cdots\) of distance thresholds, both growing doubly exponentially:
\[
    y_{j+1}=y_j^{1/(\tau-2)-o(1)}
    \qquad\text{and}\qquad
    z_{j+1}=z_j^{1/(\tau-2)-o(1)}.
\]
The first-phase layer \(A_{1,j}\) consists of vertices whose weights satisfy \(y_{j-1}\le w_v<y_j\), while the second-phase layer \(A_{2,j}\) consists of vertices whose distances to \(t\) satisfy \(z_{j-1}\le D_v<z_j\). The exact construction in the full proof uses two error parameters to handle the transition between constant and asymptotic scales. The resulting boundary layers contribute only an additive \(\Oh{f_0(n)} = \oh{\log\log n}\) term and do not affect the leading constant in \Cref{thm:length}.

We can now state a simplified trajectory lemma in the form needed below. The full technical version is proved in \Cref{appendix:geo_traj_thm}.

\begin{lemma}[Trajectory lemma, simplified form]\label{sim_traj_lemma}
Let \(w_0\) and \(D_0\) be sufficiently large constants, and let \(M:=\min\{w_0,D_0\}\). Let \(\om{1} \leq f_0(n) \leq \oh{\log\log n}\) be any growing function. Suppose that the geometric path \(P\) first enters
\(\overline{V}(w_0,D_0)\) at a vertex \(u_1\). Then, with probability
\(1-\Oh{M^{-\Omega(1)}}\), the following holds:

\(P\) follows the layered trajectory: it visits the layers in the prescribed order, visiting each first-phase layer and each second-phase layer at most once, except possibly the layer \(A_{2,\star}\), which may be visited twice. Moreover, before leaving the main part, \(P\) reaches a vertex \(u_\ell\) s.t.
\[
	\mathbb{E}_{<D_0}\big[\,\abs{N(u_{\ell}) \cap V_{<D_0}}\,\big] = \Om{M^{\Om{1}}},
\]
where the expectation is w.r.t. uncovering the vertices whose distance to target is at most \(D_0\), conditioned on the position and weight of \(u_\ell\). The number of hops from \(u_1\) to \(u_\ell\) is at most
\[
	\frac{1+\oh{1}}{\abs{\log(\tau-2)}} \left( \log\log_{w_0}(D_{u_1}^d) + \log\log_{D_0}(D_{u_1}) \right) 
	+\Oh{f_0(n)}.
\]
In particular, since \(D_{u_1}\le n^{1/d}\), this is at most \( \frac{2+\oh{1}}{\abs{\log(\tau-2)}}\log\log n\).
\end{lemma}

The main idea to prove the trajectory lemma is to establish the following property. 
Consider \(P_i\) which is the geometric path induced by the first \(i\) layers. Assuming that the \(i^{th}\) layer is not \(A_{2,\star}\), then, with sufficiently high probability, either
\begin{enumerate}
\item \(P_i\) contains no vertex from the first \(i\) layers, or
\item the first \(v \in  P_i\) contained in the \(i^{th}\) layer has a neighbor in some subsequent layer that is closer to \(t\) than every uncovered neighbor of \(v\) in the first \(i\) layers.
\end{enumerate}
In the case where the \(i^{th}\) layer is \(A_{2,\star}\), the same property holds with respect to either the first or the second vertex of \(P_i\) contained in \(A_{2,\star}\).

This is the mechanism by which the geometric choice forces the path to leave the current layer. Counting the layers induced by the doubly exponentially growing weight and distance thresholds gives the bound for the number of hops. The expectation condition involving neighbors of \(u_\ell\) is the preparation for routing to enter the end phase.

With this construction, we circumvent problematic dependencies and the trajectory lemma implies the desired length bound in \Cref{thm:length}. In \Cref{subsubsec:main_thm}, we combine this with the start and end phases to obtain the constant success probability in \Cref{thm:main}.

\subsubsection{Where to Expect Neighbors}\label{subsubsec:where_neighbor}

We now give the technical ingredients ensuring that geometric routing progresses to subsequent layers rather than remaining within the current one. Specifically, we collect the statements and describe the proof ideas that bound the weights or distances of the next neighbor chosen by the geometric routing protocol. The complete proofs are given in \Cref{appendix:where_neighbor}.

Recall the vertex sets:
\begin{align*}
V_1 \coloneqq \SetB[\Big]{ v\in V}{D_v \geq R_v(\eps_1)}\quad \text{ and} \quad
V_2 &\coloneqq \SetB[\Big]{ v\in V}{D_v \leq R_v(\eps_1)},
\end{align*}
where
\(R_v(\eps) \coloneqq ( w_v ^{1+\gamma(\eps)} )^{1/d} \approx w_v^{(\tau-1)/d(\tau-2)}\) is the expected distance from \(v\) to its furthest neighbor, and \(\eps_1 = \eps_1(\alpha,\tau)\) is a fixed constant which will be chosen sufficiently small.

We will see the weight of the next vertex \(u\) typically increases by a fixed exponent; whereas if \(v\in V_2\), the routing tends to remain within \(V_2\) and the trajectory is dominated by decreases in distance, as described in \Cref{subsec:typical_traj_geometric}.

\paragraph*{Analysis of each hop in the first phase}
Suppose \(v\in V_1\) is the current vertex. Its weight \(w_v\) is low in the sense that the current distance to the target \(D_v\) is larger than the expected distance to its furthest neighbor. In order to describe a trajectory with doubly exponential weight growth, we classify the vertices in \(V\) which accelerate the routing as desired. 
Let \(\zeta_1 \coloneqq \frac{3}{2}\) in the case \(\alpha = \infty\), and \(\zeta_1 \coloneqq \max\{\frac{2\alpha-1}{2\alpha+2-2\tau},\frac{3}{2}\}\) for case \(1<\alpha<\infty\). For \(v \in V_1\) and \(\eps = \eps(n)>0\), we define:
\begin{equation}\label{first_good/bad}
\begin{aligned}
V_1^+(v,\eps) &\coloneqq \SetB*{u\in V}{w_u \geq w_v^{\gamma(\zeta_1\eps)}\,\text{ and }\, D_u \leq D_v - \frac{R_v(\eps)}{2}},\text{ and}\\
V_1^-(v,\eps) &\coloneqq \SetB*{u\in V}{w_u < w_v^{\gamma(\zeta_1\eps)}\,\text{ and }\, D_u \leq D_v - \frac{R_v(\eps)}{2}}.
\end{aligned}
\end{equation}
The condition \(D_u \leq D_v - R_v(\eps)/2\) identifies the vertices \(u\) that are significantly closer to the target than the current vertex \(v\). The following lemma states that \(v \in V_1\) have a large expected number of good neighbors (i.e.\ \(u \in N(v) \cap V^+(v,\eps)\)); whereas the expected number of bad neighbors (i.e.\ \(u \in N(v) \cap V^-(v,\eps)\)) is small.
Notice that for \(\eps = \eps_1\), every vertex \( v\in V_1\) satisfies \(D_v \geq R_v(\eps_1)\). In this case, the statement can be applied for all \(v \in V_1\).

\begin{restatable}[\textbf{Neighborhoods in first phase}]{lemma}{FirstPhaseLemma}\label{1stphase}
Let \(0<\eps \leq \eps_1\) be sufficiently small and \(v \in V_1\) be such that \(D_v\geq R_v(\eps)\). And let \(w_1(\eps) \coloneqq c\cdot e^{d/\eps}\) for \(c\) large enough. Then
\begin{enumerate}[(i)]
	\item \(\E*{\,\abs{ N(v) \cap V^+(v,\eps)}\,} = \Om[\big]{ w_v^\eps}\), and
	\item if in addition \(w_v \geq w_1(\eps)\), then \(\E*{\,\abs{ N(v) \cap V^-(v,\eps)}\,} = \Oh[\big]{w_v^{-\Omega(\eps)}}.\)
\end{enumerate}
\end{restatable}

The idea of the proof is to investigate an annulus around \(v\) whose inner and outer radii are both on the order of \(R_v(\eps)\). This allows us to lower (resp. upper) bound the number of candidates of good (resp. bad) neighbors of \(v\).
Therefore, as the routing proceeds to a neighbor that is significantly closer to target \(t\), with high probability (w.h.p.) it also hops to a vertex with higher weight, instead of staying at low-weight vertices.

\paragraph*{Analysis of each hop in the second phase}\label{subsec:2ndphase}
Once the routing reaches a vertex in \(V_2\), it enters the second phase. The vertices in \(V_2\) have relatively large weight so that \(v\in V_2\) has neighbors at a distance roughly \(D_v\). We assume \(D_v >1\) throughout this section. This is valid since we are in the main part of routing, where it is guaranteed that \(D_v>D_0\), where \(D_0\) is a constant chosen sufficiently large. The goal now is to decrease the distance to \(t\) doubly exponentially at every step of the routing. Furthermore, once we reach a vertex in \(V_2\), we aim to remain within \(V_2\) throughout the rest of the process.

The desired behavior is almost true for the entire second phase, except for the very first step upon entering it. To be more precise, suppose the routing reaches the first vertex \(v\in V_2\) in the second phase, and let \(u \in N(v)\) be the next vertex selected by the routing protocol. Ideally, we might hope that \(D_u \approx D_v^{\tau-2}\), analogous to the behavior in the first phase. However, this is not the case. In fact, we have \(D_u \approx D_v^{(1-\eta)(\tau-1)}\), where \(\eta\in [\tfrac{\tau-2}{\tau-1-\eps_1}, \tfrac{1}{\tau-1-\eps_1}]\) is a random variable determined by the relation \(w_v = \Th{D_v^{d\eta}}\). Here, \(\eta\) captures the ratio between the weight \(w_v\) and the distance \(D_v\) when the routing first enters the second phase at vertex \(v\).
That said, from \(u\) onward, we will show that the distance to the target decreases roughly by a fixed exponent \(\tau-2\) with each subsequent hop toward the target.

The idea behind analyzing the second phase is as follows: Suppose \(v\in V_2\) is the current vertex, we search for the smallest annulus around the target \(t\) that contains a neighbor of \(v\). More precisely, given \(v\in V_2\) and \(\eps\), consider an annulus centered at \(t\) with outer radius \(r\) and inner radius \(r^{1-\Oh{\eps}}\), and we let \(r\) be the smallest value such that this annulus contains a neighbor of \(v\). By the minimality of \(r\), this annulus contains the neighbor of \(v\) that is closest to \(t\), and no neighbors of \(v\) lie inside the inner circle of radius \(r^{1-\Oh{\eps}}\). In other words, there are neighbors of \(v\) at approximately distance \(r\) from \(t\), but not much closer. Hence, the distance from the next vertex selected by the routing protocol to \(t\) will be of order \(r\).

To define the minimal annulus centered at \(t\) that contains a neighbor of \(v\). Suppose the current vertex \(v\in V_2\) is such that \(w_v = \Th{D_v^{d\eta}}\) for some \(\eta \geq \frac{\tau-2}{\tau-1-\eps_1} > \frac{\tau-2}{\tau-1}\). We may assume \(\eta\) is bounded away from \(1\) with a neglectable error probability as shown in the following:

\begin{restatable}{lemma}{HeavyTailLemma}\label{heavytail}
For \(2<\tau <3\), and \(\eps>0\) small enough, let \(V_{\geq D_0} \coloneqq \SetB{v\in V}{D_v \geq D_0}\), then
\[\Prob*{\forall\,v\in V_{\geq D_0}\,,\,w_v \leq \Th*{D_v^{d(1-\eps)}}} \geq 1- D_0^{-\Om{\eps}}.\]
\end{restatable}

\Cref{heavytail} ensures that we may assume \(\eta \leq 1-\eps\), at a cost of an error probability that is polynomially small in \(\eps\). We adopt this assumption for the remainder of this discussion. From this point on, whenever we encounter an expression of the form \( D_v ^{1-\eta -x}\), with \(x\) denoting a small quantity, we implicitly assume that \(x\) is chosen small enough such that the exponent \(1-\eta-x\) remains positive.

To define the radius of the minimal annulus, let \(r_v = r(v,\eps) \coloneqq D_v^{(1-\eta+\eps)(\tau-1)}\), where \(\eps>0\) is small enough such that \((1-\eta+\eps)(\tau-1)<1\), which is possible since \((1-\eta)(\tau-1) < \tau-2\) is bounded away from \(1\). We will also assume this throughout the rest of the discussion whenever \(\eps\) is chosen sufficiently small. Moreover, let us define \(D_1(\eps) \coloneqq c \cdot e^{d/\eps}\) to be a constant depending on \(\eps\) with \(c\) chosen sufficiently large. Then, \((1-\eta+\eps)(\tau-1)<1\) implies that we can assume \(r_v < \frac{1}{2} D_v\) when \(D_v \geq  D_1(\eps)\).

Finally, let us consider the annulus \(\mathcal{A}_{v}\) centered at \(t\), with outer radius \(r_v = D_v^{(1-\eta+\eps)(\tau-1)}\), and inner radius \(r_I = D_v^{(1-\eta-\eps)(\tau-1)} = r_v^{1-\Oh{\eps}}\). Note that the area of \(\mathcal{A}_{v}\) is \(\Th*{r_v^d}\). We shall see that \(\mathcal{A}_{v}\) is chosen such that with high probability, there is a neighbor of \(v\) in the annulus \(\mathcal{A}_{v}\), but no neighbors of \(v\) in its inner circle.

In fact, let \(\mathcal{I}_\eps\) denote the inner circle of \(\mathcal{A}_\eps\), whose the area is \(\Th{r_I^d}\). Then we have:
\begin{restatable}[\textbf{Neighborhoods in \(\mathcal{I}_{\eps}\)}]{lemma}{InnerCircleLemma}\label{2_inner}
Let \(v\in V_2\) with \(w_v = \Th{D_v^{d\eta}}\), and let \(\eps>0\) be sufficiently small. If \(D_v \geq D_1(\eps)\), then 
\[\Prob*{N(V)\cap V(\mathcal{I}_{\eps}) = \emptyset} \geq 1 - D_v^{-\Om{\eps}}.\]
\end{restatable}

Next, we show that the annulus \(A_v\) contains neighbors of \(v\). Crucially, we lower bound their weights to prevent the weight from dropping significantly, thereby ensuring that the process does not revert to the first phase. In fact, rather than merely verifying that the next vertex \(u\) satisfies the minimum weight condition for \(V_2\) (i.e. \(w_u \geq D_u^{d(\tau-2)/(\tau-1-\eps_1)}\)), we establish a strictly stronger bound. This tighter bound not only ensures \( u\in V_2\) but also reveals deeper insights into the routing behavior in the second phase.

To this end, analogous to the first phase, we define the sets of good and bad vertices for the second phase.
Let \(v\in V_2\) be such that \(w_v = \Th{D_v^{d\eta}}\), and let \(\eps = \eps(n)>0\) be sufficiently small. Let \(\zeta_2 \coloneqq 1\) for \(\alpha = \infty\), and \(\zeta_2 \coloneqq \max\{\frac{2\tau-1}{2\alpha+2-2\tau},1\}\) for \(1<\alpha<\infty\).  We define:
\begin{equation}\label{second_good/bad}
\begin{aligned}
V_2^+(v,\eps) &\coloneqq \SetB*{ u\in V}{w_u\geq D_u^{\tfrac{d(1-\eta-\zeta_2\eps)}{(\tau-1)(1-\eta+\eps)}} = D_u^{\tfrac{d}{\tau-1}\left(1-\tfrac{(\zeta_2+1)\eps}{1-\eta+\eps}\right)}},\text{ and}\\
V_2^-(v,\eps) &\coloneqq \SetB*{ u\in V}{w_u < D_u^{\tfrac{d(1-\eta-\zeta_2\eps)}{(\tau-1)(1-\eta+\eps)}} = D_u^{\tfrac{d}{\tau-1}\left(1-\tfrac{(\zeta_2+1)\eps}{1-\eta+\eps}\right)}}.
\end{aligned}
\end{equation}

Similar to the first phase, the following lemma shows that the probability of \(v\in V_2\) having a good neighbor in \(\mathcal{A}_{v}\) is high, while the probability of such \(v\) having a bad neighbor in \(\mathcal{A}_{v}\) is low. Therefore, as the routing proceeds, w.h.p. it hops to a neighbor with distance roughly \(r_v\) to the target \(t\), but not significantly closer. Moreover, the weight does not drop substantially, which ensures that the routing process stays within \(V_2\).

\begin{restatable}[\textbf{Neighborhoods in \(\mathcal{A}_v\)}]{lemma}{AnnulusLemma}\label{2_annulus}
Let \(v\in V_2\) with \(w_v = \Th{D_v^{d\eta}}\), and let \(\eps>0\) be sufficiently small.
If \(D_v \geq D_1(\eps)\), then
\begin{enumerate}[(i)]
	\item \(\Prob*{N(V)\cap V_2^+(v,\eps) \cap V(\mathcal{A}_{v})  \neq \emptyset } \geq 1 - e^{-D_v^{\Om{\eps}}}\);
	\item \(\Prob*{N(V)\cap V_2^-(v,\eps) \cap V(\mathcal{A}_{v}) = \emptyset } \geq 1 - D_v^{-\Om{\eps}}\).
\end{enumerate}
\end{restatable}

Suppose \(v\in V_2\), and \(u \in N(v)\) is the next vertex selected by the routing protocol. By \Cref{2_inner,2_annulus}, with high probability we obtain
\[D_v^{(1-\eta-\eps)(\tau-1)} \leq D_u \leq D_v^{(1-\eta+\eps)(\tau-1)}.\]
Recall that if \(v\) is the first vertex that the routing reaches upon entering the second phase, then \( \eta(v)\) is a random variable lying in the interval \([\tfrac{\tau-2}{\tau-1-\eps_1}, \tfrac{1}{\tau-1-\eps_1}]\).
From this point on, as shown in \Cref{2_annulus}, w.h.p. each subsequent vertex is not a bad neighbor of its predecessor.
In other words, we have the following corollary:
\begin{corollary}\label{cor:2nd_phase_decay}
Except for the very first vertex that the routing visits upon entering the second phase, every vertex \(v\) chosen by the routing protocol satisfies
\( w_v \geq D_v^{\tfrac{d}{\tau-1}\left(1 - \frac{(\zeta_2+1)\eps}{1-\eta+\eps} \right)},\)
where the quantity \((\frac{\zeta_2+1}{1-\eta+\eps}\cdot \eps)\) can be made arbitrarily small. In particular, this implies that \(v\in V_2\), and the next neighbor \(u\) satisfies \(D_u \leq D_v^{(\tau-2)(1+\Oh{\eps})}\), i.e.\ the distance to the target decreases roughly by an exponent of \(\tau-2\) with each hop thereafter.
\end{corollary}

\subsubsection{Proof of \Cref{thm:main}}\label{subsubsec:main_thm}
We now explain how the preceding ingredients in \Cref{subsubsec:layers,subsubsec:where_neighbor} imply that for all initial choices of \(s\) and \(t\), geometric routing succeeds with constant probability. The argument is achieved in three steps, analyzing the start phase, the main part, and the end phase.

Before the formal proof, we clarify the successful outcome that marks the completion of each phase and serves as the starting condition for the next:
\begin{enumerate}
    \item \textbf{Start phase:} Starting from an arbitrary source \(s\), the goal is to show that the path successfully initiates its journey. This means it either enters \(V_{< D_0}\) directly or finds a suitable starting vertex \(u_1 \in \overline{V}(w_0, D_0)\) to initiate the main part.
	\item \textbf{Main part:} Starting from \(u_1\), the goal is to show that the path successfully navigates the main  part and arrives at a vertex \(u_{\ell}\) that has in expectation \(\Omega(1)\) neighbors in \(V_{< D_0}\). This is precisely \Cref{sim_traj_lemma}.
	\item \textbf{End phase:} The goal is to show that the path successfully reaches the target t once it is positioned to enter \(V_{<D_0}\). This can happen either after traversing the main part to reach a vertex \(u_{\ell}\) as mentioned, or by finding such a neighbor directly during the start phase.
\end{enumerate}

Throughout the process, each step succeeds with a non-zero constant probability. However, these probabilities are not large enough to allow a union bound over all error events. Instead, we compute probabilities conditioned on the success of the previous step.

\MainTheorem*
\begin{proof}
Let us choose large enough constants \(w_0 \geq w_1(\eps_1)\) and \(D_0 \geq D_1(\eps_1)\) such that the preconditions of  \Cref{sim_traj_lemma} are met.

\proofsubparagraph{Start phase:}
First, we analyze the initiation of the routing process. We will show that the path successfully begins its journey with constant probability. More formally, we define the successful start event:

Let \(\mathcal{E}_s\) be the event that the geometric path starting at \(s\) eventually visits a vertex \(v\) that satisfies at least one of the following conditions:
\begin{itemize}
    \item \(v\) is in the end zone, i.e.\ \(D_v < D_0\); or
    \item \(v\) has \(\Om{1}\) expected neighbors in \(V_{< D_0}\); or
    \item \(v\) is a suitable starting vertex for the main part, i.e.\ \(v \in \overline{V}(w_0, D_0)\).
\end{itemize}
We now prove that \(\Prob{\mathcal{E}_s} = \Om{1}\) by analyzing all possible initial scenarios for the source \(s\).
\begin{enumerate}
    \item Case \(D_s < D_0\): The event \(\mathcal{E}_s\) is immediately satisfied with probability \(1\) by taking \(v=s\).  In this case, the region \(V_{< D_s}\) (in fact the whole graph except \(s\) and \(t\)) remains unrevealed.
    \item Case \(D_s \geq D_0\): We analyze the marginal expectation \(\mathbb{E}_{< D_0}[N(s) \cap V_{<D_0}]\) after fixing \(s\) and before uncovering \(V_{<D_0}\).
    \begin{itemize}
        \item[(1)] \(\mathbb{E}_{< D_0}[N(s) \cap V_{<D_0}] \geq \frac{1}{2}\): With probability \(\Om{1}\), \(s\) has at least one neighbor in \(V_{< D_0}\). If so, the best neighbor (closest to \(t\)) will be in \(V_{< D_0}\), and \(\mathcal{E}_s\) is again satisfied by taking \(v=s\). In this case, the routing process is ready to enter the end phase, with \(V_{< D_0}\) remains unrevealed.
        \item[(2)] \(\mathbb{E}_{< D_0}[N(s) \cap V_{<D_0}] <\frac{1}{2}\): With probability \(\Om{1}\), \(s\) has no neighbor in \(V_{< D_0}\). Conditioned on this, we consider the event where \(s\) has a unique neighbor \(s'\) in \(V_{\geq D_0}\), and such that \(w_{s'} \geq w_0\). Note that if this holds, \(s'\) is the best neighbor of \(v\), and  \(\mathcal{E}_s\) is satisfied by taking \(v = s'\).
        
        To see this happens with probability  \(\Om{1}\), note that \(\deg_{V_{\geq D_0}}(s) = \abs{N(s) \cap V_{\geq D_0}}\) is Poisson binomially distributed with expectation \(\Th{w_s} = \Th{1}\). Therefore, with probability \(\Om{1}\), \(\abs{N(s) \cap V_{\geq D_0}} =1\), there is a unique \(s' \in N(s) \cap V_{\geq D_0}\).

Moreover, the degree distribution of a uniformly random neighbor of \(s\) follows a power-law with parameter \(\tau-1\). Therefore, with probability \(\Om{1}\), the unique \(s'\) has weight \(w_{s'} \geq w_0\).
    \end{itemize}
\end{enumerate}
Since every scenario leads to the event \(\mathcal{E}_s\) with at least a constant probability, we conclude that \(\Prob{\mathcal{E}_s} = \Om{1}\).

\proofsubparagraph{Main part of routing process:}
Our goal is to show that, conditioned on \(\mathcal{E}_s\), the geometric path eventually enters \(V_{< D_0}\). In particular, we focus on the case where the start phase results in a vertex \(u_1 \in \overline{V}(w_0, D_0)\), and show that the routing proceeds to a vertex \(u_{\ell}\) which we expect to have neighbors in \(V_{< D_0}\). More formally, we consider the event
\[
\mathcal{E}_m : \text{the routing on } G[V_{\geq D_0}] \text{ finds } u_{\ell} \text{ with } \mathbb{E}_{<D_0}[\abs{N(u_{\ell}) \cap V_{<D_0}}] = \Om{1}.
\]
By \Cref{sim_traj_lemma}, we have \(\Prob{\mathcal{E}_m \mid \exists\, u_1 \in \overline{V}(w_0, D_0)} = \Om{1}\).

\proofsubparagraph{End phase:}
We now analyze the end phase, conditioned on the success of the prior stages. This means we condition on \(\mathcal{E}_s\) holding, and additionally on \(\mathcal{E}_m\) holding in the case where the geometric path entered the main part. Then, we either have \(D_s < D_0\), or the routing reached a vertex \(v\) such that \(\mathbb{E}_{<D_0}[\abs{N(v) \cap V_{<D_0}}] = \Om{1}\).

To show that the routing successfully reaches \(t\) with probability \( \Om{1}\), we first consider the latter case. Since the vertices in the region \(V_{< D_0}\) remains unrevealed, by properties of Poisson point process, \(|N(v) \cap V_{<D_0}|\) is Poisson binomially distributed with expectation \(\Om{1}\). Hence, following the same logic as the start phase, with probability \(\Om{1}\), \(v\) has a unique neighbor \(v' \in V_{<D_0}\) with a sufficiently large constant weight. Since \(w_{v'}\) is a sufficiently large constant and \(D_{v'} < D_0\), \(v'\) is then connected to \(t\) with probability \(\Om{1}\). The independence of these two events means the total probability of finding \(t\) via the two-hop path through \(v'\) is \(\Om{1}\).

In the case where \(D_s < D_0\), the region \(V_{<D_s}\) remains unrevealed. Again, \(\abs{N(s) \cap V_{<D_s}}\) is Poisson binomially distributed with expectation \(\Om{1}\). By repeating the above argument with \(D_0\) replaced by \(D_s\), it follows that the probability of finding \(t\) via an analogous two-hop path from \(s\) is also \(\Om{1}\).

In summary, we have shown that each phase of the routing succeeds with at least constant probability conditioned on previous stages being successful. Hence, geometric routing succeeds with a probability of \(\Om{1}\).
\end{proof}



\bibliography{geometric_routing_arXiv}

\appendix
\section{Proofs of \Cref{subsubsec:where_neighbor}}\label{appendix:where_neighbor}

\FirstPhaseLemma*
\begin{proof} Let \(0<\eps\leq\eps_1\) be such that \(\gamma(\eps) >1\). In the following, \(c_1\), \(c_2\) refer to the constants appearing in the definition of \(p_{uv}\), as given in \eqref{eq:EP1} and \eqref{eq:EP2}.

\proofsubparagraph{Good neighbors:} To prove the first statement, we want to lower-bound the expected number of neighbors of \(v\) in \(V^+(v,\eps)\). To do so, we count the number of vertices lying in an annulus around \(v\)  whose inner and outer radii are \(\Th{R_v(\eps)}\), located in the same direction from \(v\) as \(t\), and whose weights are high enough to connect to \(v\) with constant probability.
To be more precise, let us define \(A(v,\eps)\) to be:
\[
A(v,\eps) \coloneqq
\left\{
   u \in V \ \ \middle|\,
  \begin{aligned}
    & \text{(1)}: w_u \geq \Max*{ \frac{R_v(\eps)^d}{c_1 w_v}}{\frac{R_v(\eps)^d}{w_v}};\\
    & \text{(2)}: \frac{R_v(\eps)}{2} \leq \norm{x_u - x_v} \leq R_v(\eps);\\
    & \text{(3)}: D_u \leq D_v - \frac{R_v(\eps)}{2}.
  \end{aligned}
\right\}
\]
Condition (1) and Condition (2) guarantee that \(c_1 w_u w_v \geq R_v(\eps)^d \geq \norm{x_u - x_v}^d\). Thus for \(u\in A(v,\eps)\) we have \(p_{uv} = \Th{1}\), i.e.\ \(u\in N(v)\) with constant probability.
 Moreover, Condition (1) also implies \(w_u \geq w_v^{1+\gamma(\eps)}/w_v \geq w_v^{\gamma(\zeta_1 \eps)}\). Together with Condition (3), we have \( A(v,\eps) \subseteq V^+(v,\eps)\).

Now, it suffices to lower bound \(\E{\abs{ A(v,\eps) }}\), as it differs from \(\E{\abs{ N(v) \cap V^+(v,\eps) }}\) by at most a constant factor. To show the expected number of vertices in \(A(v,\eps)\) is large, we first analyze the volume of its underlying region, namely, the intersection of the annulus defined by (2) and the ball defined by (3), under the \(\infty\)-norm. Note that \(D_v \geq R_v(\eps_1) \geq R_v(\eps)\), a simple geometric argument yields
\[
\textrm{Vol}\,\left( 
\left\{
x\in\mathcal{X} \,\middle|\, 
\begin{aligned}
    & \frac{R_v(\eps)}{2}\leq \norm{x-x_v} \leq R_v(\eps);\\
    & \norm{x-x_t} \leq D_v- \frac{R_v(\eps)}{2}.
  \end{aligned}
\right\}
 \right) \geq \left(\frac{R_v(\eps)}{2}\right)^d.
\]
Hence, the volume of the underlying region of \( A(v,\eps)\) is \(\Th{R_v(\eps)^d}\), and
\begin{equation}\label{eq:1st_good}
\begin{aligned}
\E{\,\abs{ A(v,\eps) }\,}
&= \Th*{R_v(\eps)^d \cdot \int_{\Th{R_v(\eps)^d/w_v}}^{\infty} \Theta(w^{-\tau}) dw }\\
&= \Th*{ R_v(\eps)^{d(2-\tau)} w_v^{\tau-1} }\\
&= \Th*{ w_v^{\frac{\tau-1-\eps}{\tau-2}(2-\tau)} w_v^{\tau-1} }\\
&= \Th[\big]{w_v^{\eps}}.
\end{aligned}
\end{equation}
This yields the first result.

\proofsubparagraph{Bad neighbors:} For the second statement, we want to upper-bound the expected number of bad neighbors of \(v\).
Notice that for \(u \in V^-(v,\eps)\), we have
\[\norm{x_u - x_v} \geq \abs[\Big]{\norm{x_v - x_t} - \norm{x_u - x_t}} = D_v - D_u \geq \frac{R_v(\eps)}{2}.\]

Now, assuming that \(w_v \geq w_1(\eps) \coloneqq c \cdot e^{d/\eps}\), for \(c\) chosen to be large enough.
For the case \(\alpha = \infty\), this guarantees that \(w_v \geq (c_2 2^d)^{\frac{\tau-2}{(\zeta_1-1)\eps}}\). Therefore, the condition \(w_u \leq w_v^{\gamma(\zeta_1 \cdot \eps)}\) implies that
\begin{align*} c_2 w_u w_v &\leq c_2 \cdot w_v^{1+\gamma(\zeta_1 \eps)} \leq 2^{-d} w_v^{1+\gamma(\zeta_1\eps)+\frac{(\zeta_1-1)\eps}{\tau-2}}\\
 &= 2^{-d} w_v^{1+\gamma(\eps)} = (R_v(\eps)/2)^d \leq \norm{x_u - x_v}^d,
\end{align*}
where we use \(w_v \geq (c_2 2^d)^{\frac{\tau-2}{(\zeta_1-1)\eps}}\) in the second inequality.
Hence, \(v\) has deterministically no neighbor in \(V^-(v,\eps)\) when \(\alpha = \infty\).

For the case \(1 < \alpha < \infty\), given that \(w_v\geq w_1(\eps)\), by the same reasoning we have \(c_1 w_u w_v \leq c_2 w_u w_v \leq \norm{x_u - x_v}^d\). Then the connection probability of \(u\) and \(v\) is
\[
p_{uv} =
\Th*{
\Min*{1}{\frac{w_u w_v}{w_{\min}\cdot\norm{x_u - x_v}^d}}^\alpha
}= \Th*{ \left(\frac{w_uw_v}{\norm{x_u - x_v}^d}\right)^\alpha }.
\]
Since \(\norm{x_u - x_v} \geq R_v(\eps)/2\), the expected number of bad neighbors is upper-bounded by:
\begin{align*}
\E{\,\abs{N(v) \cap V^-(v,\eps) }\,}
&\leq
\int_{\frac{R_v(\eps)}{2}}^{n^{1/d}} \Th*{x^{d-1}}dx
\int_{w_{min} = 1}^{w_v^{\gamma(\zeta_1 \cdot \eps)}}\Th*{w^{-\tau}} \cdot \Th*{\left(\frac{w w_v}{x^d}\right)^{\alpha}}dw\\
&= \Oh*{
 w_v^{\alpha} \cdot w^{\alpha+1 -\tau} \bigg|_{1}^{w_v^{\gamma(\zeta_1\cdot\eps)}}\cdot
\int_{R_v(\eps)}^{n^{1/d}} x^{d-d\alpha-1}dx
}\\
&= \Oh*{w_v^{\alpha+ (\alpha+1-\tau)\cdot \frac{1-\zeta_1\cdot\eps}{\tau-2}} \cdot (R_v(\eps))^{d-d\alpha}}\\
&= \Oh*{w_v^{\alpha+ (\alpha-1+2-\tau)\frac{1-\zeta_1\cdot\eps}{\tau-2}} \cdot w_v^{(1+\frac{1-\eps}{\tau-2})(1-\alpha)}}\\
&= \Oh*{w_v^{\frac{(\alpha-1)\eps-(\alpha+1-\tau)\zeta_1\eps}{\tau-2}}}
\end{align*}
Notice that \(\frac{(\alpha-1)\eps-(\alpha+1-\tau)\zeta_1\eps}{\tau-2} = -\Om*{\eps}\), as \(\alpha > \tau-1\) and \(\zeta_1 > \frac{\alpha-1}{\alpha+1-\tau}\) by definition. Hence, the second claim follows.
\end{proof}

Before proving the lemmata for the second phase, we first prove a technical lemma that characterizes how the maximum weight scales with the area of a region in GIRGs. This will allow us to estimate the distribution of heavy vertices in the annulus and the inner circle.

\begin{lemma}\label{maxweight}
Let \(\mathcal{A} \subseteq \mathcal{X}\) be a region of area \(A\), and \(Z_x \coloneqq \abs*{\{v \in V(\mathcal{A}) | w_v > x\}}\). Suppose \(A>1\). For any \(\eps>0\), we have
\begin{enumerate}[(i)]
	\item \(\Prob[\Big]{ Z_{A^{1/(\tau-1)-\eps}} \geq \Om{A^{(\tau-1)\eps}} } \geq 1- e^{-A^{\Om{\eps}}}\);
	\item \(\Prob[\Big]{ Z_{A^{1/(\tau-1)+\eps}} = 0 } \geq 1- A^{-\Om{\eps}}\),
\end{enumerate}
where the hidden constant in \(\Om{A^{(\tau-1)\eps}}\) is chosen sufficiently small. And all hidden constants are uniform over all \(A\).
\end{lemma}
\begin{proof}
Let \(N = \abs{V(\mathcal{A})}\) be the number of vertices in \(\mathcal{A}\). We know that \(N \sim \textrm{Poisson}(A)\), by a Chernoff's bound we have
\[
\Prob*{ \abs{ N-A } \geq \frac{A}{2}} \leq 2e^{-\frac{A}{12}}.
\]
Recall that the weights of the vertices follow a power-law with exponent \(\tau\). Therefore, \(Z_x \sim \textrm{Bin}(N,p_x)\), with \(p_x = \Th{w^{1-\tau}}\). For the first statement, plug in \(x = \Th{A^{1/(\tau-1)-\eps}}\), for simplicity we drop the Theta notation in \(Z_x\) and write \(Z_{A^{1/(\tau-1)-\eps}}\). 
Notice that we have
\[
\E*{Z_{A^{1/(\tau-1)-\eps}}} = N \cdot \Th*{A^{(\frac{1}{\tau-1}-\eps)(1-\tau)}} = \Th*{N A^{-1+(\tau-1)\eps}}.
\]
As we have shown, with probability at least \(1 - 2e^{\frac{-A}{12}}\), \(N\) and \(A\) are of the same order.
Therefore, with the same probability, the number of vertices in \(\mathcal{A}\) with weight at least \(x = \Th{A^{1/(\tau-1)-\eps}}\) is binomially distributed with expectation \(\mu = \Th{A^{(\tau-1)\eps}}\). By a lower-tail bound for the binomial distribution, we have
\[
\Prob*{ Z_{A^{1/(\tau-1)-\eps}} \leq \frac{\mu}{2} } \leq e^{-\frac{\mu}{8}}.
\]
Note that the exponent of R.H.S. is \( \tfrac{\mu}{8} = -\Th{A^{\Om{\eps}}}\), which can be simplified to \(-A^{\Om{\eps}}\) for \(A>1\).
Hence, for sufficiently small constant \(\delta>0\), we have
\[
\Prob*{ Z_{A^{1/(\tau-1)-\eps}} \geq \delta A^{(\tau-1)\eps} } \geq 1- e^{-A^{\Om{\eps}}}.
\]
For the second statement, plug in \(x = \Th{A^{1/(\tau-1)+\eps}}\), again we simply write \(Z_{A^{1/(\tau-1)+\eps}}\) for \(Z_x\) in this case. Analogously, with probability at least \(1- e^{-A^{\Om{1}}}\), \(Z_{A^{1/(\tau-1)+\eps}}\) is binomially distributed with expectation \(\Th{A^{-(\tau-1)\eps}} = A^{-\Om{\eps}}\).
By Markov’s inequality, the probability that there exists a vertex with weight heavier than \(\Th{A^{1/(\tau-1)+\eps}}\) in \(\mathcal{A}\) is
\[
\Prob*{ Z_{A^{1/(\tau-1)+\eps}} \geq 1} \leq A^{-\Om{\eps}}.
\]
Combining the above, the error probability is bounded by \(e^{-A^{\Om{1}}} + A^{-\Om{\eps}} \leq A^{-\Om{\eps}}\). Hence,
\[\Prob*{Z_{A^{1/(\tau-1)+\eps}} = 0 } \geq 1- A^{-\Om{\eps}}.\]
\end{proof}

\HeavyTailLemma*
\begin{proof}
For \(i\geq 1\), let \(B_i\) be an annulus centered at \(t\) with radii \(2^{i-1} D_0\) and \(2^{i} D_0\). Notice that for any \(v \in B_i\), \(w_v \leq \Th*{ (2^{i-1}D_0)^{d(1-\eps)}   }\) implies \(w_v \leq \Th*{D_v^{d(1-\eps)}}\). Moreover, we have \(\textrm{Vol}(B_i) = \Th*{(2^{i-1}D_0)^{d}}\). By \Cref{maxweight}, the probability that there exist a vertex \(v \in B_i\) with weight \(w_v > \Th*{(2^{i-1}D_0)^{\frac{d}{\tau-1}+\eps}}\) is upper bounded by \(\textrm{Vol}(B_i)^{-\Om{\eps}} \). Since \(\tau>2\), we have \( \frac{d}{\tau-1}+\eps < d(1-\eps)\), for \(\eps>0\) sufficiently small. Hence
\begin{align*}
\Prob*{\exists\,v\in B_i\,,\,w_v > \Th*{(2^{i-1}D_0)^{d(1-\eps)}}} &\leq \Prob*{\exists\,v\in B_i\,,\,w_v > \Th*{(2^{i-1}D_0)^{\frac{d}{\tau-1}+\eps}}}\\ &\leq (2^{(i-1)}D_0)^{-d\Om{\eps}},
\end{align*}
where the hidden constants are uniform over all \(i\).
Take a union bound over all possible \(i\), the error probability is dominated by the first term, and thus
\begin{align*}
\sum_{i=1}^{\frac{1}{d}\log n}\Prob*{\exists\,v\in B_i\,,\,w_v > \Th*{(2^{i-1}D_0)^{d(1-\eps)}}}
&\leq \sum_{i=1}^{\frac{1}{d}\log n} (2^{(i-1)}D_0)^{-d\Om{\eps}} \leq D_0^{-\Om{\eps}}.
\end{align*}
Hence,
\[
\Prob*{\forall\,v\in V_{\geq D_0}\,,\,w_v \leq \Th*{D_v^{d(1-\eps)}}} \geq 1- D_0^{-\Om{\eps}}.
\]
\end{proof}

\InnerCircleLemma*
\begin{proof}
In the following, \(c_1\), \(c_2\) refer to the constants appearing in the definition of \(p_{uv}\), as given in \eqref{eq:EP1} and \eqref{eq:EP2}. Let \(w_v = c_v D_v^{d\eta}\) for some constant \(c_v\). Define
\[
Z_I \coloneqq \abs[\big]{ \SetB{u \in V(\mathcal{I}_{\eps})}{w_u \geq D_v^{d(1-\eta)-\eps'} }},
\]
where \(\eps'\) is any small constant such that \(\eps'> d\eps\).
Note that \(D_v^{d(1-\eta)-\eps'} = \Th{(r_I^d)^{1/(\tau-1)-(\eps'-d\eps)}}\), where \(\eps'-d\eps>0\) by assumption. Hence, by \Cref{maxweight} we have
\begin{equation}\label{eq_incircle}
\Prob*{Z_{I} =0} \leq 1- r_I^{-d\Om{\eps}} = 1 - D_v^{-\Om{\eps}}.
\end{equation}
Therefore, with probability at least \(1 - D_v^{-\Om{\eps}}\), we have \(w_u \leq D_v^{d(1-\eta)-\eps'}\) for all \(u \in V(\mathcal{I}_{\eps})\).
Note that for any vertex \(u\) in the inner circle \(\mathcal{I}_{\eps}\) , we have
\[\norm{x_u - x_v} \geq \abs{ \norm{x_v - x_t} - \norm{x_u - x_t} } = D_v - r_I > D_v/2.\]

For the case \(\alpha = \infty\), let \(u \in V(\mathcal{I}_{\eps})\) be such that \(w_u \leq D_v^{d(1-\eta)-\eps'}\). Assuming that \(D_v \geq D_1(\eps) \coloneqq c\cdot e^{d/\eps}\) for \(c\) large enough, we can assume \(D_v \geq (c_2c_v 2^d)^{1/\eps'}\). Then
\[
c_2 w_u w_v \leq c_2 D_v^{d(1-\eta)-\eps'} \cdot c_v D_v^{\eta} \leq (D_v/2)^d \leq \norm{x_u - x_v}^d,
\]
Thus with probability \(1 - D_v^{-\Om{\eps}}\), \(v\) has no neighbors in the circle \(\mathcal{I}_{\eps}\).

For the case \(1 < \alpha < \infty\), by the same assumption we have \(c_1 w_u w_v \leq c_2 w_u w_v \leq \norm{x_u-x_v}^d\) when \(D_v\) is large enough. Then the expected number of neighbors of \(v\) in \(\mathcal{I}_{\eps}\) with weight at most \(D_v^{d(1-\eta)-\eps'}\) is upper bounded by
\begin{align*}
&\Th{r_I^d} \cdot \int_{1}^{D_v^{d(1-\eta)-\eps'}} \Th*{w^{-\tau}} \cdot \Oh*{\left(\frac{w w_v}{(D_v/2)^d}\right)^{\alpha}} dw\\
&= \Th*{D_v^{d(1-\eta-\eps)(\tau-1)-d\alpha+ d\alpha \eta} \cdot w^{\alpha+1-\tau}\bigg|_{1}^{D_v^{d(1-\eta)-\eps'}} }\\
&= \Oh*{D_v^{d(1-\eta-\eps)(\tau-1) - d\alpha(1-\eta) + (\alpha+1-\tau)(d(1-\eta)-\eps')} }\\
&= \Oh*{ D_v^{-d(\tau-1)\eps -(\alpha+1-\tau)\eps'} } = D_v^{-\Om{\eps}}
\end{align*}
By Markov's inequality, \(v\) has no such neighbors with probability \(1 - D_v^{-\Om{\eps}}\). Together with \cref{eq_incircle}, we conclude that with probability \(1 - D_v^{-\Om{\eps}}\), \(v\) has no neighbors in the inner circle \(\mathcal{I}_{\eps}\).
\end{proof}

\AnnulusLemma*
\begin{proof}
In below, \(c_1\), \(c_2\) refer to the constants appearing in the definition of \(p_{uv}\), as given in \eqref{eq:EP1} and \eqref{eq:EP2}. Let \(w_v = c_v D_v^{d\eta}\) for some constant \(c_v\).

\proofsubparagraph{Good neighbors:} 
Let us consider the set
\[
W^+ = \SetB*{u\in V(\mathcal{A}_v)}{w_u \geq \Om*{D_v^{d(1-\eta)}}}
\]
Note that for \(u\in V(\mathcal{A}_v)\), we have \(D_u \leq r_v = D_v^{(1-\eta+\eps)(\tau-1)}\) and thus
\[
w_u \geq \Om*{D_v^{d(1-\eta)}} \geq \Om[\Big]{D_u^{\frac{d(1-\eta-\zeta_2 \eps)}{(\tau-1)(1-\eta+\eps)}}}.
\]
Therefore, \(W^+ \subseteq V_2^+(v,\eps) \cap V(\mathcal{A}_v)\) given that the hidden constant in \(\Om{D_v^{d(1-\eta)}}\) is chosen large enough. Hence, it suffices to show that \(N(V) \cap W^+ \neq \emptyset\). In order to do so, let \(Z_{v^+} \coloneqq \abs{W^+}\). Note that
\[\textrm{Vol}(\mathcal{A}_v) = r_v^d-r_I^d = \Om{r_v^d} = \Om{D_v^{d(1-\eta+\eps)(\tau-1)}}.\]
By \Cref{maxweight}, we have
\[\Prob*{ Z_{v^+} \geq \Om*{D_v^{\Om*{\eps}}}} \geq 1 - e^{-D_v^{\Om{\eps}}},\]
as \(\Om{D_v^{(1-\eta+\eps)(\tau-1)^2\eps}} = \Om{D_v^{\Om{\eps}}}\).

Next, suppose the hidden constant in the definition of \(W^+\) is large enough such that \(w_u \geq \frac{2^d}{3^d c_1 c_v} D_v^{d(1-\eta)}\). Moreover, notice that for \(u \in V(\mathcal{A}_v)\), we have \(\norm{x_u - x_v} \leq D_v + D_u \leq D_v + r_v\leq 3D_v/2\).
Therefore,
\[
c_1 w_u w_v \geq c_1 \frac{3^d D_v^{d(1-\eta)}}{2^d c_1 c_v} c_v D_v^{d\eta} = (3D_v/2)^d \geq \norm{x_u - x_v}^d,
\]
showing that \(p_{uv} = \Th{1}\).

To sum up, as the number of such heavy \(u\) is \(\Om{D_v^{\Om{\eps}}} = D_v^{\Om{\eps}}\) for \(D_v>1\), with probability \(1 - e^{-D_v^{\Om{\eps}}}\).
Since each of these vertices connect to \(v\) with a constant probability \(p_{uv}\), thus with probability at least \(1 - e^{-D_v^{\Om{\eps}}}\), we have
\[
\Prob*{N(V) \cap W^+ \neq \emptyset} \geq 1 - (1-p_{uv})^{{Z_{v^+}}} \geq 1 - e^{-D_v^{\Om{\eps}}}.
\]
Hence, \(\Prob*{N(V)\cap V_2^+(\eps) \cap V(\mathcal{A}_{v})  \neq \emptyset } \geq 1 - e^{-D_v^{\Om{\eps}}}\).

\proofsubparagraph{Bad neighbors:} Suppose \(u \in V_2^-(v,\eps) \cap V(\mathcal{A}_{v})\) is such that \(w_{u}\leq D_u^{\frac{d(1-\eta-\zeta_2\eps)}{(\tau-1)(1-\eta+\eps)}}\).  Moreover, we have \(\norm{x_{u}-x_v} \geq D_v - r_v \geq D_v/2\) and \(D_u \leq D_v^{(\tau-1)(1-\eta+\eps)}\).

For \(\alpha=\infty\), assuming that \(D_v \geq D_1(\eps) \coloneqq c\cdot e^{d/\eps}\) for \(c\) large enough, we can assume \(D_v \geq (c_2c_v 2^d)^{\frac{1}{d\eps}}\). Note that \(\zeta_2=1\), so
\[
c_2 w_{u} w_v \leq c_2 D_u^{\frac{d(1-\eta-\zeta_2\eps)}{(\tau-1)(1-\eta+\eps)}} \cdot c_v D_v^{d\eta} \leq c_2c_v D_v^{d(1-\eps)} \leq (D_v/2)^d \leq \norm{x_u-x_v}^d.
\]
Thus, \(u \) is deterministically not a neighbor of \(v\).

For \(1< \alpha<\infty\), by the same assumption \(D_v\geq D_1(\eps)\) and also \(\zeta_2\geq 1\), we can show that \(c_1 w_{u} w_v\leq c_2 w_{u} w_v\leq \norm{x_u-x_v}^d\). Thus, \(p_{uv} = \Oh*{\left( \frac{w_{u}w_v}{\norm{x_{u}-x_v}^d} \right)^\alpha}\).
Then the expected number of \(u\in N(v) \cap V(\mathcal{A}_{v})\) such that \(w_{u}\leq D_u^{\frac{d(1-\eta-\zeta_2\eps)}{(\tau-1)(1-\eta+\eps)}} \leq D_v^{d(1-\eta-\zeta_2\eps)}\) is upper bounded by
\begin{align*}
&\Th{r_v^d} \cdot \int_{1}^{D_v^{d(1-\eta-\zeta_2\eps)}} \Th*{w^{-\tau}} \cdot \Oh*{\left(\frac{w w_v}{(D_v/2)^d}\right)^{\alpha}} dw\\
&= \Oh*{D_v^{d(1-\eta+\eps)(\tau-1)-d\alpha+ d\alpha \eta} \cdot w^{\alpha+1-\tau}\bigg|_{1}^{D_v^{d(1-\eta-\zeta_2\eps)}} }\\
&= \Oh*{D_v^{d(1-\eta+\eps)(\tau-1) - d\alpha(1-\eta) + d(\alpha+1-\tau)(1-\eta-\zeta_2\eps)} }\\
&= \Oh*{ D_v^{d(\tau-1)\eps -d(\alpha+1-\tau)\zeta_2\eps} } = D_v^{-\Om{\eps}}.
\end{align*}
The last equality holds as \(\zeta_2 > \frac{\tau-1}{\alpha+1-\tau}\) by construction. Moreover, for any \(C>0\), we can simplify \(C \cdot D_v^{\Om{\eps}}\) into \(D_v^{\Om{\eps}}\) when \(D_v>1\). Finally, by Markov's inequality,
\[\Prob*{N(V)\cap V_2^-(v,\eps) \cap V(\mathcal{A}_{v})  = \emptyset } \geq 1 - D_v^{-\Om{\eps}}.\]
\end{proof}

\section{The Full Layer Argument for \Cref{thm:length}}\label{appendix:geo_traj_thm}
Here we present the technical layer argument and prove \Cref{thm:length}, which bounds the length of the geometric path. The analysis is intricate and builds on the lemmata developed in \Cref{subsubsec:where_neighbor,appendix:where_neighbor}, which identify where suitable neighbors can be expected.

Our approach is to formalize the intuitive behavior of the geometric path by tracking the progress of the routing process with layers and analyzing the probabilistic guarantees at each stage.
The main technical ingredient is the full trajectory lemma, whose core is the layer construction from \cref{subsec:geo_layers}.
Established in \cref{subsec:trajectory_lemma}, this lemma is the full technical version of \Cref{sim_traj_lemma} stated in the main text, showing that with high probability the geometric path does not deviate significantly from its typical trajectory. In particular, we track the geometric path by showing that the routing protocol forces it to leave the current layer and enter a subsequent layer. Moreover, with sufficiently high probability, the routing does not fail in any layer. This trajectory lemma directly implies \Cref{thm:length} and will be used to prove \Cref{thm:main} in \Cref{subsubsec:main_thm}.

\subsection{Layers for geometric routing}\label{subsec:geo_layers}
Recall from discussion in \cref{subsec:typical_traj_geometric}, it typically takes only few steps for the source \(s\) to reach a vertex with constant weight \(w_0\), and similarly, for a vertex located within a ball of radius \(D_0\) around \(t\) to reach the target \(t\). These are referred to as the start and end phases of the routing, respectively. 

To formalize the above, for any pair \((w,D)\) of weight and distance, we define
\[
\overline{V}(w, D) \coloneqq \SetB{v \in V_1}{\text{(1) : } w_v \geq w; \text{(2) : } D_v \geq D} \cup \SetB{ v \in V_2 }{D_v \geq D }.
\]
Suppose both \(w\) and \(D\) are constants. Then the main part of the routing process takes place within the set \(\overline{V}(w, D)\) for some suitable values of \(w,D\). In particular, we choose \(w_0\) and \(D_0\) sufficiently large so that the lemmata in previous section can be applied to \(\overline{V}(w_0, D_0)\). Moreover, let \( V_{\leq D_0} \coloneqq \SetB{v\in V}{D_v \leq D_0}\) characterizes the end phase that will be treated separately. Our primary goal is to show that the routing survives this main part \(\overline{V}(w_0, D_0)\) with sufficiently high probability, and eventually reaches \(V_{\leq D_0}\) which is a small ball around target \(t\).

We now define the layers for geometric routing, the constructions are based on the weights for the first phase, and on distances for the second phase. Let us start by repeating some technical definitions. For all \(\eps>0\), we denote \(\gamma(\eps) \coloneqq \frac{1-\eps}{\tau-2}\). Moreover, \(\eps_1 = \eps_1(\alpha, \tau)>0\) is a fixed constant, chosen sufficiently small throughout the proof. It also classifies the two classes of vertices:
\[
V_1 \coloneqq \SetB{v\in V}{D_v \geq R_v(\eps_1)} \quad\text{and}\quad V_2 \coloneqq \SetB{v\in V}{D_v \leq R_v(\eps_1)}.
\]
Furthermore, we defined \(w_1(\eps) = \Oh{e^{d/\eps}}\), \(D_1(\eps) = \Oh{e^{d/\eps}}\), and the auxiliary constants \(\zeta_1, \zeta_2\) where
\[
\zeta_1(\alpha, \tau) \coloneqq
\begin{cases}
    \max\left\{\frac{2\alpha - 1}{2\alpha + 2 - 2\tau}, \frac{3}{2}\right\} & \text{if } 1 < \alpha < \infty,\\
    \frac{3}{2} & \text{if } \alpha = \infty.
\end{cases}
\]
\[
\zeta_2(\alpha, \tau) \coloneqq
\begin{cases}
    \max\left\{\frac{2\tau - 1}{2\alpha + 2 - 2\tau}, 1\right\} & \text{if } 1 < \alpha < \infty,\\
    1 & \text{if } \alpha = \infty.
\end{cases}
\]
The notions of good and bad neighbors for each phase are then defined as:
\[
\begin{aligned}
V_1^+(v,\eps) &\coloneqq \SetB*{u\in V}{w_u \geq w_v^{\gamma(\zeta_1\eps)}\,\text{ and }\, D_u \leq D_v - \frac{R_v(\eps)}{2}},\text{ and}\\
V_1^-(v,\eps) &\coloneqq \SetB*{u\in V}{w_u < w_v^{\gamma(\zeta_1\eps)}\,\text{ and }\, D_u \leq D_v - \frac{R_v(\eps)}{2}}.
\end{aligned}
\]
\[
\begin{aligned}
V_2^+(v,\eps) &\coloneqq \SetB*{ u\in V}{w_u\geq D_u^{\tfrac{d(1-\eta-\zeta_2\eps)}{(\tau-1)(1-\eta+\eps)}} = D_u^{\tfrac{d}{\tau-1}\left(1-\tfrac{(\zeta_2+1)\eps}{1-\eta+\eps}\right)}},\text{ and}\\
V_2^-(v,\eps) &\coloneqq \SetB*{ u\in V}{w_u < D_u^{\tfrac{d(1-\eta-\zeta_2\eps)}{(\tau-1)(1-\eta+\eps)}} = D_u^{\tfrac{d}{\tau-1}\left(1-\tfrac{(\zeta_2+1)\eps}{1-\eta+\eps}\right)}}.
\end{aligned}
\]

We analyze the structure of the geometric path in the main part by partitioning the set \(\overline{V}(w_0, D_0)\) into layers. The idea is to show that, with sufficiently high probability, there exists a node \(u_\ell\) on \(P\) which has neighbors in \(V_{\leq D_0}\), such that until reaching \(u_\ell\), the path \(P\) visits every layer at most once, except possibly the special layer \(A_{2,\star}\), which may be visited twice. The special layer \(A_{2,\star}\) is the first layer in the second phase that the path reaches when entering the second phase.

We define two classes of layers which correspond to the two phases: the layers \(A_{1,j}\) divide the area \(V_1 \cap \overline{V}(w_0, D_0)\) and are defined via weights, whereas the layers \(A_{2,j}\) divide the area \(V_2 \cap \overline{V}(w_0, D_0)\) and are defined via distances.

Assuming that \(f_0(n) = \oh{\log \log n}\), and recall that we set \(\eps_1 = \Th*{1}\) to be a constant. Then, let
\(\eps_2 \coloneqq (\log \log f_0(n))^{-1} = \oh{1}\) be a slowly decaying parameter, where we may assume \(\zeta_1\eps_2 < \eps_1\) and \(\zeta_2\eps_2 < \eps_1\). Finally, we set the landmarks that serve as thresholds for the layer definitions:
\[
w_0' \coloneqq w_0^{(\gamma(\zeta_1 \eps_1)^{f_0(n))})} \quad \text{and} \quad D_0' \coloneqq D_0^{(\gamma(\zeta_2\eps_1)^{f_0(n))})}.
\]
For later reference, we note that since \(\gamma(\zeta_i \eps_1) > 1\) for \(i=1,2\),
\begin{align*}
&w_0'^{\Om{\eps_2}} = w_0^{\Om*{\gamma(\zeta_1 \eps_1)^{f_0(n)}/ \log \log f_0(n)}} = \om*{w_0^{\Om{1}}},\\
&D_0'^{\Om{\varepsilon_2}} = D_0^{\Om*{\gamma(\zeta_2 \eps_1)^{f_0(n)}/ \log \log f_0(n)}} = \om*{D_0^{\Om{1}}}.
\end{align*}

\paragraph*{Layers by weight:} Define a sequence \(y_0 \coloneqq w_0 < y_1 < \cdots < y_j < \cdots\) of weights which grows doubly exponentially. More precisely, let
\[
y_{j+1} =
\begin{cases}
y_j^{\gamma(\zeta_1 \eps_1)} & \text{if } y_j < w'_0, \\
y_j^{\gamma(\zeta_1 \eps_2)} & \text{if } y_j \geq w'_0.
\end{cases}
\]
For the set \(V_1' := \left\{ v \in \overline{V}(w_0, D_0) \mid w_v^{1+\gamma(\zeta_1\eps_2)} \leq D_v^d \right\}\), partition the vertices in \(V_1'\) into the following layers:
\[
A_{1,j} \coloneqq \left\{ v \in V_1' \mid y_{j-1} \leq w_v < y_j \right\} \quad \forall j \geq 1.
\]
Note that by definition, \(V_1' \subset V_1\) and no vertex in \(V_1\) has weight larger than \(n^{(1+\gamma(\eps_1))^{-1}}\). Thus we only need to consider sets \(A_{1,j}\) such that
\(y_{j-1} \leq n^{(1+\gamma(\eps_1))^{-1}}\).
Furthermore, we introduce the additional layer that contains the remaining vertices of \((V_1 \cap \overline{V}(w_0, D_0))\setminus V_1'\):
\[
A_{1,\infty} \coloneqq \left\{ v \in \overline{V}(w_0,D_0) \,\middle|\, w_v^{1+\gamma(\eps_1)}\leq D_v^d \leq w_v^{1+\gamma(\zeta_1\eps_2)} \right\}.
\]

\paragraph*{Layers by distance:}
Define a sequence \(z_0 \coloneqq D_0 < z_1 < \cdots < z_j < \cdots\) of distances which grows doubly exponentially. More precisely, let
\[
z_{j+1} =
\begin{cases}
z_j^{\gamma(\eps_1)} & \text{if } y_j < D'_0, \\
z_j^{\gamma(\eps_2)} & \text{if } y_j \geq D'_0.
\end{cases}
\]
For the set \(V_2' := \left\{ v \in V_2 \cap \overline{V}(w_0, D_0) \mid D_v^{\gamma(\eps_2)} \leq n^{1/d}\right\}\), partition the vertices in \(V_2\) into the following layers:
\[
A_{2,j} \coloneqq \left\{ v \in V_2' \mid z_{j-1} \leq D_v < z_j \right\} \quad \forall j \geq 1.
\]
Note that we only need to consider sets \(A_{2,j}\) for which
\(z_j^{\gamma(\eps_2)} \leq n^{1/d}\).
Furthermore, we introduce the additional layer that contains the remaining vertices in \((V_2 \cap \overline{V}(w_0, D_0) )\setminus V_2'\):
\[
A_{2,\infty} \coloneqq \left\{ v \in V_2\cap\overline{V}(w_0,D_0) \mid D_v \geq n^{1/(d \gamma(\eps_2))} \right\}.
\]

\begin{remark*}
The exponent in the definition of \(y_j\) (resp. \(z_j\)) are chosen specifically such that the good neighbors in \(V_1^+(v,\eps)\) (resp.  \(V_2^+(v,\eps)\)) will leave the current layer containing \(v\). Moreover, the number of \(j\)'s such that \(y_j\) grows with \(\gamma(\zeta_1 \eps_1)\) (resp. \(z_j\) grows with \(\gamma(\eps_1)\)) in the exponent is \( \oh{\log \log n}\), which is a negligible fraction among all \(\Th*{\log \log n}\) layers.
\end{remark*}

\begin{remark*}
We set \(\eps_2 = \oh{1}\) in order to capture the tighter concentration of \(w_v^{(1-\zeta_1\eps_2)/(\tau-2)} \leq w_u \approx w_v^{1/(\tau-2)}\), as well as \(D_v^{(\tau-2)(1+\Oh{\eps_2})} \geq D_u \approx{D_v^{\tau-2}}\), when \(w_v\) and \(D_v\) are larger. Moreover, \(\eps_2\) is chosen to decay slow enough such that \( \Om{w_v^{\eps}} \geq \Om{1}\) and \( \Om{D_v^{\eps}} \geq \Om{1}\) are true for all layers. This is clearly true when \(\eps = \eps_1 =\Th{1}\), as well as for \(\eps=\eps_2\) since \(w_v\geq y_{j-1}\),  \(D_v \leq z_{j-1}\), and that \(\eps_2\) decay slower than the growth of \(y_j\) and \(z_j\).
\end{remark*}

\subsection{Trajectory lemma for the main part}\label{subsec:trajectory_lemma}
We will utilize the constructed layers to establish the trajectory lemma for geometric routing, which states that with sufficiently high probability, the routing will not fail in \(\overline{V}(w_0, D_0)\) and will traverse this set ultra-fast. 

Interestingly, a key phenomenon of the geometric path, not observed in the greedy path, is that the doubly exponential growth at rate \(1/(\tau-2)\) in weight, or decay at rate \(\tau-2\) in distance, is not guaranteed immediately after the phase transition. More precisely, let \(v\) be the first vertex of \(P\) in \(V_2\), and denote by \(A_{2,\star}\) the layer containing \(v\). Suppose that \(w_v = \Th{D_v^{d\eta}}\), then the hop from \(v\) to its next neighbor \(u\) is a mixture of weight increase and distance decrease, depending on the precise value of the random variable \(\eta\). As a result, \(u\) is likely to remain in the layer \(A_{2,\star}\). Nevertheless, this intermediate step places \(u\) in a good position, the ratio between \(w_u\) and \(D_u\) ensures doubly exponential decay in distance at rate \(\tau-2\) from \(u\) onward (see also \Cref{cor:2nd_phase_decay}). 

Before proving the trajectory lemma, we first state a technical lemma which is a consequence of the Poisson point process, that will help us handle dependent events.

\begin{lemma}\label{lem:correlation}
Let \(A_1, A_2 \subset \mathcal{X}\) and let \(v, u_1, \ldots, u_r \in \mathcal{X} \setminus (A_1 \cup A_2)\) be \(r + 1\) vertices, for some \(r \geq 1\). Then
\[
\Prob{N(v) \cap A_1 = \emptyset \mid N(u_1) \cap A_2 = \ldots = N(u_r) \cap A_2 = \emptyset} \geq \Prob{N(v) \cap A_1 = \emptyset}.
\]
\end{lemma}

\begin{proof}
The proof for general \(r\) is identical to the case \(r = 1\), so we present only the latter for readability.

Let \(A_1, A_2 \subset \mathcal{X}\) be two subsets and let \(v, u_1\) be two vertices not contained in the sets \(A_1 \cup A_2\). Suppose we know that \(u_1\) has no neighbors in \(A_2\). Clearly, conditioning on this event decreases the expected number of vertices in \(A_2\) and in every subset of \(A_2\). Hence,
\[ 
\Prob{N(v) \cap (A_1 \cap A_2) = \emptyset \mid N(u_1) \cap A_2 = \emptyset} \geq \Prob{N(v) \cap (A_1 \cap A_2) = \emptyset}. 
\] 
Furthermore, by the properties of Poisson point process, for every set \(A_3\) which is disjoint from \(A_2\), the number of vertices in \(A_2\) and \(A_3\) is independent. We put \(A_3 := A_1 \setminus A_2\) and deduce 
\[ 
\Prob{N(v) \cap (A_1 \setminus A_2) = \emptyset \mid N(u_1) \cap A_2 = \emptyset} = \Prob{N(v) \cap (A_1 \setminus A_2) = \emptyset}. 
\] 
Combining the two inequalities gives the desired property. 
\end{proof}

We now state and prove the trajectory lemma.

\begin{lemma}[Trajectory Lemma]\label{traj_lemma}
Let \(w_0 \geq w_1(\eps_1) \), \(D_0 \geq D_1(\eps_1)\), and denote \(M := \min\{w_0, D_0\}\).
Let \(f_0(n) = \om{1}\) be any growing function such that \(f_0(n) = \oh{\log\log n}\).
Let \(s\) be the starting vertex, suppose \(D_s \geq D_0\), denote by \(P\) the geometric path induced by all vertices in \(V_{\geq D_0}\). And suppose that there exists a first vertex \(u_1 \in P \cap \overline{V}(w_0, D_0)\)
Then, with probability
\[
1 - \Oh*{ M^{-\Om{1}}},
\]
there exists a subpath \(P' = (u_1, \ldots, u_\ell) \subseteq P \) starting at \(u_1\) such that:
\begin{enumerate}[(i)]
\item \(P' \subseteq \overline{V}(w_0, D_0) \).
\item Either \(P' \subseteq V_1 \), or \(P' \subseteq V_2\), or there exists a vertex \(u_\ell'\) such that \(\{u_1, \ldots, 	u_\ell'\} \subseteq V_1 \) and \(\{u_{\ell'+1}, \ldots, u_\ell\} \subseteq V_2\).
\item Let \(\{u_i, u_{i+1}, u_{i+2}\}\) be consecutive vertices on \(P' \cap V_1 \). Then \(w_{u_{i+2}} \geq w_{u_i}^{\gamma(\zeta_1 \eps_1)}\).
\item Let \(\{u_i, u_{i+1}, u_{i+2}\}\) be consecutive vertices on \(P' \cap V_2 \), and suppose \(u_i\) is not the vertex with smallest index in \(P' \cap V_2 \). Then 	\(D_{u_{i+2}} \leq D_{u_i}^{1/\gamma(\eps_1)}\).
\item The length \(\ell\) of the path \(P'\) satisfies 
\[\ell \leq \frac{2 + \oh{1}}{|\log(\tau - 2)|} \log \log n.\]
More precisely, the length of \(P'\) is upper bounded by
\[
\frac{1 + \oh{1}}{|\log(\tau - 2)|} \left( \log \log_{w_0} (D_{u_1}^d) + \log \log_{D_0} (D_{u_1}) \right) + \Oh{f_0(n)}.
\]
\item The following holds:
\[
\mathbb{E}_{<D_0}\left[ \abs{ \N(u_\ell) \cap V^+(u_\ell, \eps_1) \cap V_{<D_0} } \right] = \Om{M^{\Om{1}}},
\]
where \(\mathbb{E}_{<D_0}\) denotes the expectation w.r.t. uncovering the vertices whose distance to target is at most \(D_0\), conditioned on the position and weight of \(u_\ell\).
\end{enumerate}
\end{lemma}

Note that \Cref{traj_lemma} naturally applies to the situation that we have uncovered all vertices outside the ball \(B_{D_0}(t)\), but not the vertices inside. More precisely, properties in \textup{(i)} to \textup{(v)} are independent of \(V_{< D_0}\), and \textup{(vi)} concerns a marginal expectation after uncovering \(V_{\geq D_0}\) and before uncovering  \(V_{< D_0}\).

\begin{proof}
The layers \(A_{i,j}\) defined above classify all vertices \(v\) with distance \(D_v \geq D_0\). We order them by:
\[
A_{1,1} \prec \ldots \prec A_{1,j} \prec A_{1,j+1} \prec \ldots \prec A_{1,\infty} \prec A_{2,\infty} \prec \ldots \prec A_{2,j} \prec A_{2,j-1}
\prec \ldots \prec A_{2,1}.
\]
Let \(A_{2,\star}\) denote the first layer that \(P\) visited when entering the second phase. Alternatively, \(\star\) is the largest index  (possibly \(\star=\infty\)), such that \(P \cap A_{2,\star} \neq \emptyset\).
We want to show that with sufficiently high probability, the induced path \(P\) contains at most two vertices of \(A_{2,\star}\), and at most one vertex of every other layer \(A_{i,j}\), until reaching a vertex \(u_\ell\) that satisfies \textup{(vi)}.

\proofsubparagraph*{Definition of the events \(\mathcal{E}\) and \(\mathcal{E}_{i,j}\):}
Given the ordering of the layers, we denote by \(B_{i,j}\) the union of \(A_{i,j}\) and all previous layers. And by \(P_{i,j}\) we denote the geometric path induced by the set \(B_{i,j} \subset V\).

In the following, for a given pair \((i,j) \neq (2,\star)\), we consider the first vertex \(v \in P_{i,j} \cap A_{i,j}\), if it exists. We will show that with sufficiently high probability, the \emph{best} neighbor of \(v\), namely the one at minimal distance to \(t\), is located outside \(B_{i,j}\). Such phenomenon is desirable, and is captured by the ``good'' event described below. Note that for \((i,j) = (2,\star)\), we allow the \emph{best} neighbor of \(v\) to remain in \(B_{i,j}\), whose \emph{best} neighbor, is however, located outside \(B_{i,j}\). 
 
More precisely, let \(i \in \{1,2\}\), \(j \geq 1\). Let \(\mathcal{E}_{1,j}\) be the event that either \(P_{1,j} \cap A_{1,j} = \emptyset\), or the first vertex \(v \in P_{1,j} \cap A_{1,j}\)
\begin{itemize}
\item satisfies condition \textup{(vi)}, or
\item has at least one \emph{good} neighbor and no \emph{bad} neighbors. Then, the \emph{best} neighbor \(u\) satisfies \(u \in \overline{V}(w_0, D_0) \setminus B_{1,j}\) with \(D_{u} \leq D_v\) such that \(D_{u} < D_{u'}\) holds for all \(u' \in N(v) \cap B_{1,j}\).
\end{itemize}

Furthermore, for \(j \neq \star\), let \(\mathcal{E}_{2,j}\) be the event that either \(P_{2,j} \cap A_{2,j} = \emptyset\), or the first vertex \(v \in P_{2,j} \cap A_{2,j}\)
\begin{itemize}
\item satisfies condition \textup{(vi)}, or
\item has at least one \emph{good} neighbor and no \emph{bad} neighbors. Then, the \emph{best} neighbor satisfies \(u \in \overline{V}(w_0, D_0) \setminus B_{2,j}\), and \(w_u = \Th{D_u^{d\eta}}\) is such that \(\eta \geq \frac{1-\eps_0}{\tau-1}\), with \(\eps_0 \leq \eps\) sufficiently small.
\end{itemize}

Finally, we define \(\mathcal{E}_{2,\star}\) to be the event that the first vertex \(v \in P_{2,\star} \cap A_{2,\star}\)
\begin{itemize}
\item satisfies condition \textup{(vi)}, or
\item has at least one \emph{good} neighbor and no \emph{bad} neighbors. Then, the \emph{best} neighbor satisfies \(u \in \overline{V}(w_0, D_0) \setminus B_{2,\star}\), and \(w_u = \Th{D_u^{d\eta}}\) is such that \(\eta \geq \frac{1-\eps_0}{\tau-1}\), with \(\eps_0 \leq \eps\) sufficiently small.
\item \(N(v) \subseteq B_{2,\star}\), and any \emph{best} neighbor \(u\) of \(v\) such that \(D_{u} \leq D_{u'}\) holds for all \(u' \in N(v) \cap B_{2,\star}\), satisfies condition \textup{(vi)} or has at least one \emph{good} neighbor and no \emph{bad} neighbors.
\end{itemize}

Finally, we denote by \(\mathcal{E} \coloneqq \cap_{i,j} \mathcal{E}_{i,j}\) the intersection of all ``good'' events.

\proofsubparagraph*{Lower bound for \(\Prob{\mathcal{E}_{i,j}}\):}
Our goal is to lower-bound \(\Prob{\mathcal{E}_{i,j}}\) for every pair \((i,j)\) in order to lower-bound \(\Prob{\mathcal{E}}\). We will consider several different subcases, which correspond to different subphases of the routing by construction. In all cases let \(\eps \in \{\eps_1,\eps_2\}\) denote the same \(\eps\) as used for the definition of the considered layer. Clearly, we have \(\Prob{\mathcal{E}_{i,j} \mid P_{i,j} \cap A_{i,j} = \emptyset} =1\) for \((i,j)\neq (2,\star)\), and \(P_{2,\star} \cap A_{2,\star} \neq \emptyset\) by definition. Therefore, it suffices to consider \(\Prob{\mathcal{E}_{i,j} \mid P_{i,j} \cap A_{i,j} \neq \emptyset}\). 

\underline{\textit{Case \(i=1\) and \(j<\infty\)}}: We are in the first phase of the routing, and the layer \(A_{1,j}\) is contained in \(V_1' \subset V_1\).
Then by \Cref{1stphase} \textup{(i)}, we have \[\E{\abs{N(v) \cap V^+_1(v,\eps)}} = \Om{w_v^\eps}.\]
We do not expect many of these neighbors in \(V_{<D_0}\). In particular, we can upper bound \(\E{\abs{N(v) \cap V^+_1(v,\eps) \cap V_{<D_0}}}\) by replacing the region in \Cref{eq:1st_good} by \(V_{<D_0}\), and get
\begin{align*}
\E{\abs{N(v) \cap V^+_1(v,\eps) \cap V_{<D_0}}} &= \Oh*{ D_0^d \cdot \int_{R_v(\eps)^d/w_v}^\infty w^{-\tau} dw}\\
&= \Oh*{w_v^{\frac{(\tau-1)(-1+\eps)}{\tau-2}}}\\
&= \Oh*{w_v^{-\Om{\eps}}}.
\end{align*}
Hence,
\(\E{\abs{N(v) \cap V^+_1(v,\eps) \cap V_{\geq D_0}}} = \Om{w_v^\eps}.\) By a Chernoff's bound, using \(w_v \geq y_{j-1}\), with probability at least \(1 - e^{-\Om{y_{j-1}^\eps}}\) there exists a neighbor \(u\) of \(v\) in \(V^+(v,\eps) \cap V_{\geq D_0}\).
By the definition of \(V_1^+(v,\eps)\), \(u\) satisfies \(D_{u} \leq D_v -R_v(\eps)/2\) and \(u \notin B_{1,j}\) since \(w_u \geq w_v^{\gamma(\zeta_1\eps)} \geq y_{j-1}^{\gamma(\zeta_1 \eps)} = y_j\).

Now, we want to show that the routing indeed hops to a good neighbor \(u \in N(v)\setminus B_{1,j}\). Namely, every \(u' \in N(v) \cap B_{1,j}\) has \(D_{u'} > D_{u}\). Suppose \(u' \in B_{1,j}\) and \(D_{u'} \leq D_v - R_v(\eps)/2\). Since \(w_{u'} < y_j = y_{j-1}^{\gamma(\zeta_1\eps)} \leq w_v^{\gamma(\zeta_1\eps)}\), we must have \(u' \in V_1^-(v,\eps)\). Thus, we want to upper bound \(\E{\abs{N(v) \cap V_1^-(v,\eps)}}\), and show that with sufficiently high probability, no bad neighbors of \(v\) exist.

Note that \(v\) cannot be treated as a fixed vertex to compute this expectation, since \(v\) is defined as the first vertex in \( P_{1,j} \cap A_{1,j}\), but the induced path \(P_{1,j}\) depends on \(V_1^-(v,\eps)\). Thus \(v\) is in fact a random variable. However, exposing \(V_1^-(v,\eps)\) changes \(v\) if and only if there is some \(u' \in V_1^-(v,\eps)\) which is adjacent to an earlier vertex on \(P_{1,j}\), because then \(u'\) would be the first vertex in \( P_{1,j} \cap A_{1,j} \) and not \(v\).
Hence, to compute \(\Prob{N(v) \cap V_1^-(v,\eps) = \emptyset}\) for our random \(v\), we compute \(\Prob{N(v) \cap V_1^-(v,\eps) = \emptyset}\) for a fixed \(v\) and condition on that no earlier vertex on the path \(P_{1,j}\) has a neighbor \(u' \in V_1^-(v,\eps)\). 
By \Cref{lem:correlation}, this condition decreases the expected number of vertices in \(V_1^-(v,\eps)\), and therefore also the expected number of neighbors of \(v\) in \(V_1^-(v,\eps)\).
Hence, in order to upper bound the probability that \(N(v) \cap V_1^-(v,\eps) \neq \emptyset\), we are allowed to neglect the dependencies. Since \(y_{j-1} \geq w_1(\eps) = \Oh{e^{d/\eps}}\), by \Cref{1stphase} \textup{(ii)} and Markov's inequality we deduce
\begin{align*}
\Prob{N(v) \cap V_1^-(v,\eps) = \emptyset \mid v \in A_{1,j} \cap P_{1,j}} &\leq \Prob{N(v) \cap V_1^-(v,\eps) = \emptyset}\\
&\leq \E{\abs{N(v) \cap V_1^-(v,\eps)}}\\
& = \Oh{y_{j-1}^{-\Om{\eps}}}.
\end{align*}
Therefore, with probability at least \(1 - \Oh{y_{j-1}^{-\Om{\eps}}}\), \(v\) has no neighbor in \(V_1^-(v,\eps)\) and thus every \(u' \in N(v)\cap B_{1,j}\) has distance \(D_{u'} > D_v - R_v(\eps)/2 \geq D_{u}\) as desired. It follows
\[
\Prob{\mathcal{E}_{1,j}} \geq 1 - \Oh{ y_{j-1}^{-\Om{\eps}}}.
\]
Depending on whether \(\eps\) is \(\eps_1\) or \(\eps_2\), the exponent of the above expression is either \(-\Om{1}\) or \(-\Om{(\log \log f_0(n))^{-1}}\).

\underline{\textit{Case \(i=1\) and \(j=\infty\)}}: A lower bound for \(\Prob{\mathcal{E}_{i.\infty}}\) can be obtained in a similar way. Let \(v \in P_{1,\infty} \cap A_{1,\infty} \) be the first vertex in this extra layer, then by definition \(w_v^{1+\gamma(\eps_1)} \leq D_v^d \leq w_v^{1+\gamma(\zeta_1\eps_2)}\). Therefore, we can apply \Cref{1stphase} with \(\eps = \eps_1\) (but not with \(\zeta_1\eps_2\)). Then similarly, every vertex \(u \in V_1^+(v, \eps_1)\) satisfies
\[
w_u^{1+\gamma(\eps_1)} \geq w_v^{\gamma(\zeta_1\eps_1)\cdot(1+\gamma(\eps_1))} \geq w_v^{\gamma(\zeta_1\eps_1)\cdot\frac{1+\gamma(\eps_1)}{1+\gamma(\zeta_1\eps_2)}} D_v^d \geq D_u^d,
\]
where one can verify \(\gamma(\zeta_1\eps_1)\cdot\frac{1+\gamma(\eps_1)}{1+\gamma(\zeta_1\eps_2)}\geq 1\) when \(\eps_1\) is sufficiently small. Therefore, \(u \in V_1^+(v, \eps_1)\), which implies \(u \in V_2\). Moreover, by exactly the same argument as in the case \(i=1\) and \(j<\infty\), we upper bound \(\E{\abs{N(v) \cap V_1^-(v,\eps_1)}}\) by applying \Cref{lem:correlation,1stphase} \textup{(i)}, and deduce that
\[
\Prob{\mathcal{E}_{1,\infty}} \geq 1 - \Oh{ w_0^{-\Om{\eps_1}}}.
\]

\underline{For \(i=2\):}

We consider the event \(\mathcal{E}_2\) such that \(\forall\,v\in V_{\geq D_0}\,,\,w_v \leq \Th*{D_v^{d(1-\eps_1)}}\). Note that by taking \(\eps_1\) sufficiently small, by \Cref{heavytail} we have
\[\Prob*{\mathcal{E}_2} \geq 1- D_0^{-\Om{\eps_1}}.\]
By admitting \(\mathcal{E}_2\), we may assume \(\eta\) is bounded away from \(1\). Therefore, we can choose \(\eps'\) sufficiently small such that the inner radius \(r_I = D_v^{(1-\eta-\eps')(\tau-1)} = D_v^{\Om{\eps'}}\), and thus assume \(r_I\geq D_0\) for all but the last layer \(A_{2,1}\), as the growth in \(D_v\) dominates the decay of \(\eps'\).

\underline{\textit{Case \(i=2\) and \(1<j<\star\)}}: For \(1<j<\star\), we can again assume that the set \(P_{2,j} \cap A_{2,j}\) is non-empty, and the first vertex \( v \in P_{2,j} \cap A_{2,j}\) does not satisfy \textup{(vi)}. By \Cref{2_annulus} \textup{(i)} and the argument in the paragraph above,
\begin{align*}
\Prob*{N(v) \cap V_2^+(v,\eps') \cap V_{\geq D_0} \neq \emptyset} &\geq \Prob*{N(v) \cap V_2^+(v,\eps') \cap V_{\geq r_I} \neq \emptyset}\\
&\geq \Prob*{N(v) \cap V_2^+(v,\eps') \cap V(\mathcal{A}_v) \neq \emptyset}\\
&\geq 1-e^{-D_v^{\Om{\eps'}}},
\end{align*}
for \(\eps'\) sufficiently small. 

Suppose \( u\in N(v) \cap V_2^+(v,\eps') \cap V(\mathcal{A}_v)\). Since \(j< \star\), by \Cref{cor:2nd_phase_decay}, we can write \(w_v = \Th{D_v^{d\eta(v)}}\), where \(\eta(v) \geq \frac{1-\eps''}{\tau-1}\) for some \(\eps''\) sufficiently small. Hence, we deduce
\[
D_u \leq D_v^{(1-\eta(v)+\eps')(\tau-1)} \leq D_v^{1/\gamma(\eps)} < z_j^{1/\gamma(\eps)} = z_{j-1},
\]
where the first inequality holds by the definition of \(\mathcal{A}_v\), and the second inequality holds by taking \(\eps'\), \(\eps''\) sufficiently small. Therefore, \(u \notin B_{2,j}\).

It remains to show that the next node selected by the protocol would again satisfy desired weight condition, i.e.\ would be a good neighbor. Now suppose \(u' \in N(v)\) satisfies \(D_{u'} \leq D_v^{(1-\eta(v)+\eps')(\tau-1)}\), but \(w_{u'} = \Th{D_{u'}^{d\eta(u')}}\) is such that \(\eta(u')\) is not sufficiently close to \(\frac{1}{\tau-1}\), that is \(\eta(u') \leq \frac{1}{\tau-1}\left( 1- \frac{(\zeta_2+1)\eps}{1-\eta+\eps}\right)\). Then \Cref{lem:correlation,2_annulus} \textup{(ii)} tell us the probability that there exists no such neighbor \(u' \in N(v)\) is \(1-D_v^{-\Om{\eps}}\).
As a result,
\[\Prob{\mathcal{E}_{2,j}} \geq 1 -\Oh{z_{j-1}^{-\Om{\eps}}}.\]
Depending on whether \(\eps\) is \(\eps_1\) or \(\eps_2\), the exponent of the above expression is either \(-\Om{1}\) or \(-\Om{(\log \log f_0(n))^{-1}}\).

\underline{\textit{Case \(i=2\) and \(j=\star\)}}: The argument for \(j = \star\) works similar as \(j < \star\). The only difference is that the first vertex \( v \in P_{2,\star} \cap A_{2,\star}\) where the random variable \(\eta(v)\) in \(w_v = \Th{D_v^{d\eta(v)}}\) might not necessarily satisfy \(\eta(v) \geq \frac{1-\eps''}{\tau-1}\). In the case that \(v\) does not satisfy condition \textup{(vi)}, nor has a good neighbor outside \(B_{2,\star}\). We consider the next node \(v'\) selected by the protocol. Note that \(D_{v'} \leq D_v\) and \(v' \in B_{2,\star}\), thus \(v' \in A_{2,\star}\). We may now apply the same argument as in the case \(i=2\) and \(j<\star\) to \(v'\) instead of \(v\). Therefore,
\[\Prob{\mathcal{E}_{2,\star}} \geq 1 -\Oh{D_0^{-\Om{\eps_1}}}.\]

\underline{\textit{Case \(i=2\) and \(j=1\)}}: In this final layer, we may not simply assume \(r_I \geq D_0\). However, \Cref{2_annulus} still guarantees that
\begin{align*}
\Prob*{N(v) \cap V_2^+(v,\eps') \cap V(\mathcal{A}_v)\neq \emptyset} \geq 1-e^{-D_v^{\Om{\eps'}}},
\end{align*}
for \(\eps' <\eps_1\) sufficiently small. In this case, \(u \in N(v) \cap V_2^+(v,\eps')\) satisfies \(D_u \leq D_v^{(1-\eta(v)+\eps')(\tau-1)} \leq D_v^{1/\gamma(\eps_1)} < z_1^{1/\gamma(\eps_1)} = z_0 = D_0\). That is \(N(v) \cap V_2^+(v,\eps') = N(v) \cap V_2^+(v,\eps') \cap V_{<D_0}\), and thus 
\[
\E{\abs{V_2^+(v,\eps_1) \cap V_{<D_0}}} \geq \E{\abs{V_2^+(v,\eps') \cap V_{<D_0}}} \geq \Om{M^{\Om{1}}},
\]
we must have \textup{(vi)} holds. Hence,
\[\Prob{\mathcal{E}_{2,1}} \geq 1 - e^{-D_0^{\Om{\eps_1}}}.\]

\proofsubparagraph*{Lower bound for \(\Prob{\mathcal{E}}\):} We have computed a lower bound for \(\Prob{\mathcal{E}_{i,j}}\) for every layer \(A_{i,j}\). Combining the above, a union bound yields
\begin{align*}
\Prob{\mathcal{E}} &\geq 1 - \sum_{j\geq 1} \Prob{\mathcal{E}_{i,j}^c} - \Prob{\mathcal{E}_{1,\infty}^c} - \Prob{\mathcal{E}_{2}^c} - \sum_{j \geq 2}^{\star} \Prob{\mathcal{E}_{2,j}^c} - \Prob{\mathcal{E}_{2,1}^c} \\
&= 1 - \Oh*{ w_0^{-\Om{1}} + w_0'^{-\Om{\eps_2}} + D_0^{-\Om{1}} + D_0'^{-\Om{\eps_2}}  } \\
&= 1 - \Oh*{\min\{w_0, D_0\}^{-\Omega(1)}}.
\end{align*}

\proofsubparagraph*{Construction of \(P'\) (proves \textup{(i)}, (ii), (vi)):} 
Having proven that \(E\) occurs with sufficiently high probability, it remains to show that this event implies the desired properties. Suppose that \(E\) occurs and that \(P\) visits at least one vertex in \(\overline{V}(w_0, D_0)\), and thus there exists a first vertex \(u_1 \in P \cap \overline{V}(w_0, D_0) \) in some layer \(A_{i,j}\). We construct a subpath \(P' \subseteq P\) inductively whose existence implies \textup{(i) - (iv)}.

\underline{Suppose that \(u_1 \in V_1\):} 

There exists a maximal subpath \(P_1 = \{u_1, \dots, u_{\ell'}\} \subseteq P\) starting at \(u_1\) such that \(P_1 \subset V_1\) and no vertex before \(u_{\ell'}\) satisfies \textup{(vi)}. We start with an analysis of this subpath \(P_1\) and show by induction that every vertex of \(P_1\) is located in a subsequent layer compared to its predecessor. The induction hypothesis holds trivially for the base case \(\{u_1\}\). 

Assume that the induction hypothesis holds for \(\{u_1, \dots, u_i\}\). Then \(u_i\) is contained in a layer \(A_{1,j}\), and by assumption, \(u_i\) is the first vertex that \(P\) visits in this layer. Moreover, by the induction hypothesis it follows that \(\{u_1, \dots, u_i\}\) is also a subpath of the induced greedy path \(P_{1,j}\). As the event \(\mathcal{E}_{1,j}\) occurs, we have one of the following:
\begin{itemize}
    \item \(u_i\) satisfies \textup{(vi)}: Then set \(P'=(u_1, \dots, u_i)\) with \(u_\ell \coloneqq u_i\).
    \item \(u_i\) has a good neighbor and no bad neighbors: Let \(u_{i+1}\) be the best neighbor of \(u_i\), then \(u_{i+1} \in V_2\), or \(u_{i+1}\) is located in a layer \(A_{1,j'}\) where \(j' > j\). Since the protocol makes greedy choices, \(u_{i+1}\) follows \(u_i\) on the path \(P\).
\end{itemize}
This proves the induction hypothesis. We see that \(P_1\) traverses the layers \(A_{1,j}\) according to the ordering and visits each layer at most once. Furthermore, applying the same argument for the last vertex \(u_{\ell'} \in P_1\) shows that either \(u_{\ell'}\) satisfies \textup{(vi)}, or it has a good neighbor in a layer \(A_{2,j}\). In the former case, we would choose \(u_\ell = u_{\ell'}\) and \(P'=P_1\), and \textup{(i)}, \textup{(ii)}, \textup{(vi)} follow directly. Otherwise, we continue to find a path \(P'\) such that \(P' \cap V_1 = P_1\).

\underline{Suppose that the last vertex \(u_{\ell'} \in P_1\) does not satisfy \textup{(vi)} or \(u_1 \in V_2\):}  

There exists a first vertex \(u_{\ell'+1} \in P \cap V_2 \cap V_{>D_0}\), and a maximal subpath \(P_2 = P_1 \cup \{u_{\ell'+1}, \dots, u_\ell\} \subseteq P\) starting at \(u_1\) such that \((P_2 \setminus P_1) \subset V_2 \cap V_{>D_0}\) and no vertex before \(u_\ell\) satisfies \textup{(vi)}. We continue by induction and prove that every vertex of \(P_2\), except possibly \(u_{\ell'+2}\), is located in a subsequent layer compared to its predecessor. By assumption, the hypothesis holds for \(P_1 \cup \{u'_{\ell'+1}\}\). For \(P_1 \cup \{ u_{\ell'+1}, u_{\ell'+2}\}\), either \(u_{\ell'+2} \in A_{2,\star}\) or \(u_{\ell'+2} \in A_{2,j}\) for \(j<\star\), the induction hypothesis also holds. 

Assume that the induction hypothesis holds for \(P_1 \cup \{u'_{\ell'+1}, \dots, u_i\}\), where \(i\geq \ell'+2\). Then either \(u_i\) is the second vertex that \(P\) visited in layer \(A_{2,\star}\); or \(u_i\) is contained in a layer \(A_{2,k}\) with \(k<\star\), and \(u_i\) is the first vertex that \(P\) visits in \(A_{2,k}\). Moreover, \(P_1 \cup \{u'_{\ell'+1}, \dots, u_i\}\) is also a subpath of the induced greedy path \(P_{2,j}\) (possibly \(j=\star\)). Consider the event \(\mathcal{E}_{2,j}\) for both \(j=\star\) and \(j<\star\), we have one of the following:
\begin{itemize}
	\item \(u_i\) satisfies \textup{(vi)}: Then set \(P'=(u_1, \dots, u_i)\) with \(u_\ell := u_i\).
	\item \(u_i\) has a good neighbor and no bad neighbors: Let \(u_{i+1}\) be the best neighbor of \(u_i\). Then \(u_{i+1} \in A_{2,j'}\) for \(j' < j\). Since the protocol makes greedy choices, \(u_{i+1}\) follows \(u_i\) on the path \(P\).
\end{itemize}
This proves the induction hypothesis. We observe that \(P_2\) traverses the layers \(A_{2,j}\) according to our ordering, visiting \(A_{2,\star}\) at most twice and every other layer \(A_{2,k}\) (for \(k<\star\)) at most once. However, by repeating the argument for the last vertex \(u_\ell\) it follows that the best neighbor of \(u_\ell\) can only be in \(V_{\leq D_0}\). Then the only possibility is that \(u_\ell\) satisfies \textup{(vi)}. Hence we put \(P' = P_2\), and \textup{(i)}, \textup{(ii)}, and \textup{(vi)} again follow. In particular, we see that \(P' \cap V_1 = P_1\) as claimed above.

\paragraph*{Proof of properties \textup{(iii)}, \textup{(iv)}:}
If \(\{u_i, u_{i+1}, u_{i+2}\}\) are three subsequent vertices on the subpath \(P_1 \subset V_1\), then the weight increases at least by an exponent \(\gamma(\zeta_1\eps_1)\) between \(u_i\) and \(u_{i+2}\) since there exists at least one layer in-between containing \(u_{i+1}\). Similarly, for three subsequent vertices \(\{ u_i, u_{i+1}, u_{i+2}\}\) on the subpath \(\{ u_{\ell'+2}, ..., u_{\ell}\}\), the distance to \(t\) decreases at least by an exponent \(\gamma(\eps_1)\) between \(u_i\) and \(u_{i+2}\) as there lies at least one layer in-between containing \(u_{i+1}\).

\paragraph*{Proof of properties \textup{(v)}:}
We see that the path \(P'\) visits the special layer \(A_{2,\star}\) at most twice and every other layer at most once.
Therefore we can upper-bound the length of \(P'\) by counting the total number of potentially visited layers (differ by at most one).

By construction, the first \(f_0(n)\) layers \(A_{1,j}\) are defined via \(\varepsilon_1\), and similarly the final \(f_0(n)\) layers \(A_{2,j}\) are defined via \(\varepsilon_1\).

Let \(L_1\) denote the number of layers \(A_{1,j}\) defined via \(\eps_2 = o(1)\) which the routing potentially visited.
Observe that every vertex \(v \in P'\) satisfies \(D_v \leq D_{u_1}\). Thus, every vertex \(v \in P'\) with weight at least \(D_{u_1}^{d/(1+\gamma(\eps_1))} \geq D_{v}^{d/(1+\gamma(\eps_1))}\) belongs to \(V_2\). Therefore, we only need to count layers \(A_{1,j}\) which contain vertices of weight at most \(D_{u_1}^{d/(1+\gamma(\eps_1))}\). Then \(L_1\) is upper-bounded by the solution of \(w_0^{\gamma(\eps_2)^{L_1}} = D_{u_1}^d\). It follows
\[
L_1 \leq \frac{\log \log_{w_0}(D_{u_1}^d)}{\log \gamma(\eps_2)} \leq \frac{\log \log_{w_0}(D_{u_1}^d)}{\log((\tau-2)^{-1/(1+\oh{1})})} = \frac{1+\oh{1}}{\abs{\log(\tau-2)}} \log \log_{w_0}(D_{u_1}^d).
\]

Let \(L_2\) denote the number of layers \(A_{2,j}\) defined via \(\eps_2 = o(1)\) which the routing potentially visited. In order to upper-bound \(L_2\), note that the distance to target \(t\) is always decreasing, thus every vertex \(v \in P'\) must satisfies \(D_v \leq D_{u_1}\). Clearly, this implies that we only need to consider layers \(A_{2,j}\) where \(z_{j-1} \leq D_{u_1}\). Therefore, \(L_2\) is upper-bounded by the solution of \(D_0^{\gamma(\eps_2)^{L_2}} = D_{u_1}\), and we obtain
\[
L_1 \leq \frac{\log \log_{D_0}(D_{u_1})}{\log \gamma(\eps_2)} \leq \frac{\log \log_{D_0}(D_{u_1})}{\log((\tau-2)^{-1/(1+\oh{1})})} = \frac{1+\oh{1}}{\abs{\log(\tau-2)}} \log \log_{D_0}(D_{u_1}).
\]

The length of \(P'\) is at most \(L_1 + L_2 + \Oh*{f_0(n)}\), which yields the desired result. Moreover, we have \(D_{u_1} \leq n^{1/d}\) (which is also the typical case), thus the length \(\ell\) of \(P'\) satisfies:
\[
\ell \leq \frac{2+\oh{1}}{\abs{\log(\tau-2)}} \log \log n.
\]
This proves \textup{(v)}.
\end{proof}

\end{document}

%% file: macrosetup.tex
\newcommand{\N}{\ensuremath{\mathbb{N}}}

\DeclarePairedDelimiter\abs{\lvert}{\rvert}
\DeclarePairedDelimiter\norm{\lVert}{\rVert}

\renewcommand{\epsilon}{\ensuremath{\varepsilon}}
\newcommand{\eps}{\epsilon}

\renewcommand{\phi}{\ensuremath{\varphi}}

\DeclarePairedDelimiter{\Paren}{(}{)}
\DeclarePairedDelimiter{\Bracket}{[}{]}
\DeclarePairedDelimiter{\Brace}{\{}{\}}

\NewDocumentCommand{\Oh}{som}{%
  \ensuremath{\mathcal{O}\IfBooleanTF{#1}{%
    \Paren*{#3}
  }{%
    \IfNoValueTF{#2}{%
      \Paren{#3}
    }{%
      \Paren[#2]{#3}
    }%
  }%
}}

\NewDocumentCommand{\Om}{som}{%
\ensuremath{\Omega\IfBooleanTF{#1}{%
    \Paren*{#3}%
  }{%
    \IfNoValueTF{#2}{%
      \Paren{#3}%
    }{%
      \Paren[#2]{#3}%
    }%
  }%
}}

\NewDocumentCommand{\Th}{som}{%
\ensuremath{\Theta\IfBooleanTF{#1}{%
    \Paren*{#3}%
  }{%
    \IfNoValueTF{#2}{%
      \Paren{#3}%
    }{%
      \Paren[#2]{#3}%
    }%
  }%
}}

\NewDocumentCommand{\oh}{som}{%
\ensuremath{o\IfBooleanTF{#1}{%
    \Paren*{#3}%
  }{%
    \IfNoValueTF{#2}{%
      \Paren{#3}%
    }{%
      \Paren[#2]{#3}%
    }%
  }%
}}

\NewDocumentCommand{\om}{som}{%
\ensuremath{\omega\IfBooleanTF{#1}{%
    \Paren*{#3}%
  }{%
    \IfNoValueTF{#2}{%
      \Paren{#3}%
    }{%
      \Paren[#2]{#3}%
    }%
  }%
}}

\NewDocumentCommand{\Prob}{som}{%
\ensuremath{\mathbb{P}
  \IfBooleanTF{#1}{%
    \Bracket*{#3}%
  }{%
    \IfNoValueTF{#2}{%
      \Bracket{#3}%
    }{%
      \Bracket[#2]{#3}%
    }%
  }%
}}

\NewDocumentCommand{\E}{som}{%
\ensuremath{\mathbb{E}
  \IfBooleanTF{#1}{%
    \Bracket*{#3}%
  }{%
    \IfNoValueTF{#2}{%
      \Bracket{#3}%
    }{%
      \Bracket[#2]{#3}%
    }%
  }%
}}

\NewDocumentCommand{\Max}{s o m m}{%
\ensuremath{\max
  \IfBooleanTF{#1}{%
    \Brace*{#3,\, #4}%
  }{%
    \IfNoValueTF{#2}{%
      \Brace{#3,\, #4}%
    }{%
      \Brace[#2]{#3,\, #4}
    }%
  }%
}}

\NewDocumentCommand{\Min}{s o m m}{%
\ensuremath{\min
  \IfBooleanTF{#1}{%
    \Brace*{#3,\, #4}%
  }{%
    \IfNoValueTF{#2}{%
      \Brace{#3,\, #4}%
    }{%
      \Brace[#2]{#3,\, #4}
    }%
  }%
}}

\NewDocumentCommand{\Exp}{som}{%
\ensuremath{\exp
  \IfBooleanTF{#1}{%
    \Paren*{#3}%
  }{%
    \IfNoValueTF{#2}{%
      \Paren{#3}%
    }{%
      \Paren[#2]{#3}%
    }%
  }%
}}

\NewDocumentCommand{\Log}{som}{%
\ensuremath{\log
  \IfBooleanTF{#1}{%
    \Paren*{#3}%
  }{%
    \IfNoValueTF{#2}{%
      \Paren{#3}%
    }{%
      \Paren[#2]{#3}%
    }%
  }%
}}

\NewDocumentCommand{\Var}{som}{%
\ensuremath{\mathrm{Var}
  \IfBooleanTF{#1}{%
    \Bracket*{#3}%
  }{%
    \IfNoValueTF{#2}{%
      \Bracket{#3}%
    }{%
      \Bracket[#2]{#3}%
    }%
  }%
}}

\NewDocumentCommand{\SetBuilder}{s o m m}{%
\ensuremath{\IfBooleanTF{#1}{%
    \Brace*{\,#3 \mid #4\,}%
  }{%
    \IfNoValueTF{#2}{%
      \Brace{\,#3 \mid #4\,}%
    }{%
      \Brace[#2]{\,#3 \mid #4\,}%
    }%
  }%
}}
\newcommand{\SetB}{\SetBuilder}